\documentclass{amsart}
\usepackage{amssymb,mathtools}
\usepackage[british]{babel}
\usepackage{enumitem}
\usepackage[latin1]{inputenc}
\usepackage{url}
\usepackage{hyperref}
\usepackage{tikz-cd}

\newtheorem{theorem}{Theorem}
\numberwithin{theorem}{section}
\newtheorem{corollary}[theorem]{Corollary}
\newtheorem{lemma}[theorem]{Lemma}
\newtheorem{proposition}[theorem]{Proposition}
\theoremstyle{definition}
\newtheorem{definition}[theorem]{Definition}

\newtheorem{example}[theorem]{Example}

\newcommand{\supp}{\operatorname{supp}}
\newcommand{\rng}{\operatorname{rng}}
\newcommand{\leqf}{\leq^{\operatorname{fin}}}

\newcommand{\en}{\operatorname{en}}

\newcommand{\seq}{\operatorname{Seq}}
\newcommand{\trace}{\operatorname{Tr}}
\newcommand{\hig}{\operatorname{Hig}}
\newcommand{\emb}{\operatorname{Emb}}
\newcommand{\T}{\mathcal T}
\newcommand{\rca}{\mathbf{RCA}}
\newcommand{\aca}{\mathbf{ACA}}
\newcommand{\cac}{\mathbf{CAC}}
\newcommand{\pica}{\mathbf{\Pi^1_1\text{-}CA}}
\newcommand{\poz}{\operatorname{PO_0}}
\newcommand{\po}{\operatorname{PO}}
\newcommand{\lo}{\operatorname{LO}}
\newcommand{\loz}{\operatorname{LO_0}}
\newcommand{\id}{\operatorname{id}}

\title[A uniform Kruskal theorem]{Minimal bad sequences are necessary for\\ a uniform Kruskal theorem}
\author{Anton Freund, Michael Rathjen and Andreas Weiermann}

\begin{document}

\begin{abstract}
The minimal bad sequence argument due to Nash-Williams is a powerful tool in combinatorics with important implications for theoretical computer science. In particular, it yields a very elegant proof of Kruskal's theorem. At the same time, it is known that Kruskal's theorem does not require the full strength of the minimal bad sequence argument. This claim can be made precise in the framework of reverse mathematics, where the existence of minimal bad sequences is equivalent to a principle known as $\Pi^1_1$-comprehension, which is much stronger than Kruskal's theorem. In the present paper we give a uniform version of Kruskal's theorem by relativizing it to certain transformations of well partial orders. We show that $\Pi^1_1$-comprehension is equivalent to our uniform Kruskal theorem (over~$\mathbf{RCA}_0$ together with the chain-antichain principle). This means that any proof of the uniform Kruskal theorem must entail the existence of minimal bad sequences. As a by-product of our investigation, we obtain uniform proofs of several Kruskal-type independence results.
\end{abstract}


\maketitle

\section{Introduction}

Recall that a partial order consists of a set~$X$ and a binary relation $\leq_X$ on $X$ that is reflexive, transitive, and antisymmetric. An infinite sequence $x_0,x_1,\dots$ in~$X$ is called good if there are $i<j$ with $x_i\leq_X x_j$; otherwise it is called bad. If $X$ contains no infinite bad sequence, then it is called a well partial order (wpo). Let us write $\mathcal T$ for the set of finite trees. To obtain a partial order, we agree that $T\leq_{\mathcal T} T'$ holds if there is an injection $f:T\to T'$ that preserves infima. Kruskal's theorem~\cite{kruskal60} asserts that $\mathcal T$ is a well partial order. Note that there are several variants of Kruskal's theorem, in particular for structured and labelled trees.

A particularly short and transparent proof of Kruskal's theorem was given by Nash-Williams~\cite{nash-williams63}. Assuming that the theorem fails, the idea is to consider a bad sequence $T_0,T_1,\dots$ that is minimal in the following sense: For each~$i$, if
\begin{equation*}
T_0,\dots,T_{i-1},T_i',T_{i+1}',\dots
\end{equation*}
is an infinite bad sequence, then $T_i$ has at most as many vertices as~$T_i'$. It is relatively straightforward to show that the existence of such a minimal bad sequence leads to a contradiction, which establishes Kruskal's theorem.

Reverse mathematics~\cite{simpson09} is a well-established framework in which one can compare the strength of theorems and proof methods. The basic idea is to establish implications and equivalences over a weak base theory, most often the theory $\rca_0$ of recursive comprehension. It turns out that many theorems from various areas of mathematics are equivalent to one of four principles. We will encounter the three stronger of these, which are known as arithmetical comprehension ($\aca_0$), arithmetical transfinite recursion ($\mathbf{ATR}_0$) and $\Pi^1_1$-comprehension ($\pica_0$). To establish the main result of our paper, we will extend the base theory $\rca_0$ by the chain-antichain principle $\cac$, which asserts that any infinite partial order contains an infinite chain (linear suborder) or antichain. This principle is much weaker than arithmetical comprehension and ensures that different definitions of well partial order are equivalent (see~the analysis by Cholak, Marcone and Solomon~\cite{cholak-RM-wpo}).

Due to work of Schmidt~\cite{schmidt-habil} and Friedman (presented by Simpson~\cite{simpson85}), Kruskal's theorem does not follow from arithmetical transfinite recursion. This is significant, because it shows that the incompleteness phenomenon from G\"odel's theorem does apply to important statements of core mathematics. The precise strength of Kruskal's theorem has been determined by Rathjen and Weiermann~\cite{rathjen-weiermann-kruskal}. In particular, their analysis shows that Kruskal's theorem (even with labels) is far weaker than $\Pi^1_1$-comprehension. On the other hand, Marcone~\cite{marcone-bad-sequence} has shown that $\Pi^1_1$-comprehension is equivalent to a general statement about the existence of minimal bad sequences, which is known as the minimal bad sequence lemma. From a foundational perspective, this shows that the minimal bad sequence argument does not yield the most elementary proof of Kruskal's theorem (of course, it may still yield the most elegant proof).

In the present paper we show that a uniform version of Kruskal's theorem is equivalent to $\Pi^1_1$-comprehension. In view of the previous paragraph, this means that our uniform Kruskal theorem exhausts the full strength of the minimal bad sequence lemma. To motivate our approach, we consider the least fixed point $\T W_A$ of the transformation
\begin{equation*}
X\mapsto W_A(X):=A\times X^{<\omega},
\end{equation*}
where $A$ is a given partial order and $X^{<\omega}$ denotes the set of finite sequences with entries in~$X$. Observe that $\T W_A$ is the set of finite structured trees with labels in~$A$. Indeed, an element $(a,\langle T_1,\dots,T_n\rangle)\in W_A(\T W_A)$ corresponds to the tree with root label~$a$ and immediate subtrees $T_1,\dots,T_n$. Also note that the usual embeddability relation between structured trees is related to the Higman order on~$X^{<\omega}$.  Now if $A$ is a well partial order, then $W_A(X)$ is a wpo for any wpo~$X$. Using the minimal bad sequence argument, one can infer that $\mathcal T W_A$ is a wpo as well. Hasegawa~\cite{hasegawa97} has shown that the second step of this argument applies to any transformation~$W$ of well partial orders, even when $W$ does not come from a family of transformations $W_A$ indexed by partial orders. In particular, he has given a general construction of~$\T W$ (with a suitable order relation) as a direct limit, relative to a suitable functor~$W$. Hasegawa's paper also shows that $\T W$ is relevant for theoretical computer science, as it provides a uniform foundation for the recursive path orderings used in term rewriting (cf.~the work of Dershowitz~\cite{dershowitz82}).

Working in the setting of reverse mathematics, it is not straightforward to express general statements about transformations of (countable) orders, because these transformations are class-sized (or at least uncountable) objects. In the linear case, a solution is provided by Girard's dilators~\cite{girard-pi2}, which are defined as particularly uniform functors between well orders. Due to their uniformity, dilators are determined by their values on the category of finite orders, which makes it possible to represent them in reverse mathematics. In Section~\ref{sect:formalize-po-dilators} of this paper we introduce \mbox{PO-dilators} as uniform functors between partial orders. We then define \mbox{WPO-dilators} as \mbox{PO-dilators} that map wpos to wpos. To avoid misunderstanding, dilators in the sense of Girard are called WO-dilators in the present paper. \mbox{Finally}, an LO-dilator is a transformation of linear orders that has all properties of a WO-dilator, except that it does not need to preserve well foundedness.

In Section~\ref{sect:generalized-Kruskal} we define a partial order $\T W$ for any PO-dilator~$W$ that satisfies a certain normality condition. Our construction of $\T W$ as a term system builds on previous work of Weiermann~\cite{weiermann09}. We can then formulate the following principle:
\begin{equation*}
\parbox{0.97\textwidth}{\textbf{Uniform Kruskal Theorem.} If $W$ is a normal WPO-dilator, then $\T W$ is a well partial order.}
\end{equation*}
From a logical perspective, we observe that this is a $\Pi^1_3$-statement. In contrast, a statement of the form ``if~$A$ is a wpo, then $\T W_A$ is a wpo" has complexity $\Pi^1_2$. It is known that $\Pi^1_1$-comprehension is a $\Pi^1_3$-statement that cannot be equivalent to any $\Pi^1_2$-statement.

Still in Section~\ref{sect:generalized-Kruskal}, we adapt Hasegawa's argument to show that the uniform Kruskal theorem follows from the minimal bad sequence lemma, and hence from \mbox{$\Pi^1_1$-comprehension}. In the rest of our paper we prove the converse direction, so that we get the following result (see Theorem~\ref{thm:main} below for the official statement):
\begin{equation*}
\parbox{0.97\textwidth}{\textbf{Main Result.} Over $\rca_0+\cac$, the uniform Kruskal theorem is equivalent to $\Pi^1_1$-comprehension, and hence to the minimal bad sequence lemma.}
\end{equation*}
To show that $\Pi^1_1$-comprehension follows from the uniform Kruskal theorem, we use a result of Freund~\cite{freund-equivalence,freund-categorical,freund-computable}. Inspired by Rathjen's notation system for the Bachmann-Howard ordinal (see~\cite{rathjen-weiermann-kruskal}), Freund has defined a linear order $\vartheta(D)$ for any LO-dilator~$D$. He has then shown that $\Pi^1_1$-comprehension is equivalent to the following statement:
\begin{equation*}
\parbox{0.97\textwidth}{\textbf{Bachmann-Howard Principle.} If $D$ is a WO-dilator, then $\vartheta(D)$ is a well~order.}
\end{equation*}
To deduce the main result of the present paper, it suffices to show that the uniform Kruskal theorem implies the Bachmann-Howard principle.

Recall that a function $f:X\to Y$ between partial orders is a quasi embedding if $f(x)\leq_Y f(x')$ implies $x\leq_X x'$. By a quasi embedding $\nu:D\Rightarrow W$ of an \mbox{LO-dilator}~$D$ into a PO-dilator~$W$ we shall mean a natural family of quasi embeddings $\nu_X:D(X)\to W(X)$ for all linear orders~$X$. In Section~\ref{sect:lower-bounds} of this paper we show that any quasi embedding $D\Rightarrow W$ induces a quasi embedding of the linear order $\vartheta(D)$ into the partial order~$\T W$. It follows that $\vartheta(D)$ is a well order if $\T W$ is a well partial order. In Section~\ref{sect:extend-dilator} we show how to construct a PO-dilator $W_D$ and a quasi embedding $D\Rightarrow W_D$ for a given LO-dilator~$D$. Our main technical result proves that $W_D$ preserves well partial orders if $D$ preserves well orders. Given that $W_D$ is a WPO-dilator, the uniform Kruskal theorem tells us that $\T W_D$ is a well partial order. As we have seen, it follows that $\vartheta(D)$ is a well order. This establishes the Bachmann-Howard principle and completes the proof of our main result.

The results of Section~\ref{sect:lower-bounds} also yield the following principle, which allows us to re-establish several known independence results and guides our search for new ones. This principle has certainly been known as a heuristic, which has been used in several concrete applications. However, the precise statement and proof of the general principle seem to be new.
\begin{equation*}
\parbox{0.97\textwidth}{\textbf{Uniform Independence Principle.} Consider a computable WPO-dilator~$W$ and a sound theory~$\mathbf T\supseteq\rca_0$. To show that the statement ``$\T W$ is a well partial order" is independent of~$\mathbf T$, it suffices to find a computable \mbox{WO-dilator}~$D$ such that (i)~$\mathbf T$~proves that there is a quasi embedding $D\Rightarrow W$ and (ii) the order~type of $\vartheta(D)$ is at least as big as the proof theoretic ordinal of~$\mathbf T$.}
\end{equation*}
In~\cite{freund_predicative-collapsing} the order type of~$\vartheta(D)$ has been determined for some natural WO-dilators~$D$. For example, it has been shown that  $D(X)=1+2\times X^2$ (with an appropriate linear order) leads to~$\vartheta(D)\cong\Gamma_0$. By the uniform independence principle we can now recover Friedman's result~\cite{friedman-kruskal-notes} that Kruskal's theorem for binary trees with two labels is independent of the theory~$\mathbf{ATR}_0$ of arithmetical transfinite recursion.

\section{Dilators on partial orders}\label{sect:formalize-po-dilators}

Girard~\cite{girard-pi2} has singled out uniform transformations of (well founded) linear orders, which he calls dilators. In the present section we introduce \mbox{PO-dilators} as uniform transformations of partial orders. We also show how \mbox{PO-dilators} can be represented in the setting of reverse mathematics.

The notions of (well) partial order and quasi embedding have been recalled in the introduction. We will consider them as part of the following structure:

\begin{definition}
The category PO of partial orders has the partial orders as objects and the quasi embeddings as morphisms.
\end{definition}

Recall that a quasi embedding $f:X\to Y$ is an embedding if $x\leq_X x'$ is equivalent to $f(x)\leq_Y f(x')$. We say that a functor $W:\po\to\po$ preserves embeddings if $W(f):W(X)\to W(Y)$ is an embedding whenever $f:X\to Y$ is one. If $Y$ is linear, then any quasi embedding $f:X\to Y$ is an embedding. Hence the category LO of linear orders and embeddings is a full subcategory of~PO.

Let us write $[\cdot]^{<\omega}$ for the finite subset functor on the category of sets, given by
\begin{align*}
[X]^{<\omega}&=\text{``the set of finite subsets of~$X$"},\\
[f]^{<\omega}(a)&=\{f(x)\,|\,x\in a\}\quad\text{(for $f:X\to Y$ and $a\in[X]^{<\omega}$)}.
\end{align*}
We will also apply $[\cdot]^{<\omega}$ to partial orders, omitting the forgetful functor to the underlying sets. Conversely, a subset~$a\subseteq X$ of a partial order~$X$ will often be considered as a suborder (such that the inclusion $a\hookrightarrow X$ is an embedding).

\begin{definition}\label{def:po-dilator}
A class-sized PO-dilator consists of
\begin{enumerate}[label=(\roman*)]
\item a functor $W:\po\to\po$ that preserves embeddings and
\item a natural transformation $\supp:W\Rightarrow[\cdot]^{<\omega}$ such that the support condition
\begin{equation*}
\rng(W(f))=\{\sigma\in W(Y)\,|\,\supp_Y(\sigma)\subseteq\rng(f)\}.
\end{equation*}
holds whenever $f:X\to Y$ is an embedding (not just a quasi embedding).
\end{enumerate}
If $W(X)$ is a wpo for every wpo~$X$, then $W$ is called a class-sized WPO-dilator.
\end{definition}

In the second part of this section we will show that class-sized PO-dilators can be represented by certain set-sized objects, which we call coded PO-dilators. The latter allow us to make general statements about PO-dilators without quantifying over proper classes. Note that class quantifiers are needed if one wants to state the equivalence between class-sized and coded PO-dilators. For us, this equivalence (and the notion of class-sized PO-dilator itself) will only play a heuristic role.

Let us state the corresponding definition for the linear case: A (class-sized) LO-dilator consists of a functor~$D:\lo\to\lo$ and a natural transformation $\supp:D\Rightarrow[\cdot]^{<\omega}$ such that the support condition from clause~(ii) of Definition~\ref{def:po-dilator} holds for any embedding $f:X\to Y$ of linear orders. If $D(X)$ is a well order for any well order~$X$, then $D$ is a (class-sized) WO-dilator. We point out that our \mbox{WO-dilators} coincide with Girard's dilators: The support condition ensures that any WO-dilator preserves direct limits and pullbacks, as demanded by Girard. Conversely, a functor that preserves direct limits and pullbacks can be equipped with support functions, which are unique and in particular natural (see~\cite[Remark~2.2.2]{freund-thesis} for a detailed verification). Our LO-dilators coincide with the prae-dilators considered in~\cite{freund-equivalence,freund-computable}. Finally, Girard's pre-dilators (note the different spelling) coincide with the monotone LO-dilators of Definition~\ref{def:monotone-dilators} below.

Concerning the support condition from Definition~\ref{def:po-dilator}, we observe that one inclusion is automatic: If we have $\sigma\in\rng(W(f))$, say $\sigma=W(f)(\sigma_0)$ with $\sigma_0\in W(X)$, then the naturality of supports yields
\begin{equation*}
\supp_Y(\sigma)=[f]^{<\omega}(\supp_X(\sigma_0))\subseteq\rng(f).
\end{equation*}
The other inclusion ensures that elements and inequalities in~$W(X)$ can only depend on finite suborders of~$X$: Given $\sigma,\tau\in W(X)$, we put $a:=\supp_X(\sigma)\cup\supp_X(\tau)$. Due to the support condition we can write $\sigma=W(\iota_a^X)(\sigma_0)$ and $\tau=W(\iota_a^X)(\tau_0)$, where $\iota_a^X:a\hookrightarrow X$ is the inclusion. Since the PO-dilator~$W$ preserves embeddings, we learn that $\sigma\leq_{W(X)}\tau$ is equivalent to $\sigma_0\leq_{W(a)}\tau_0$. As we shall see, this ensures that class-sized PO-dilators are essentially determined by their restrictions to the category of finite partial orders.

Important examples of PO-dilators arise if we take $W(X)$ to be some collection of finite graphs (e.\,g.~lists or trees) with labels in~$X$. We will see that the following PO-dilators are connected to Higman's lemma.

\begin{example}\label{ex:po-dilator-aca}
Given a partial order~$Z$, we define a class-sized PO-dilator $W_Z$ as follows: For each partial order~$X$, the underlying set of $W_Z(X)$ is given as
\begin{equation*}
W_Z(X)=1+Z\times X=\{0\}\cup\{(z,x)\,|\,z\in Z\text{ and }x\in X\}.
\end{equation*}
To order this set we declare that the only inequalities are $0\leq_{W_Z(X)}0$ and
\begin{equation*}
(z,x)\leq_{W_Z(X)}(z',x')\quad\text{for $z\leq_Z z'$ and $x\leq_X x'$}.
\end{equation*}
Given a quasi embedding $f:X\to Y$, we define $W_Z(f):W_Z(X)\to W_Z(Y)$ by
\begin{equation*}
W_Z(f)(0)=0\quad\text{and}\quad W_Z(f)((z,x))=(z,f(x)).
\end{equation*}
One readily checks that this turns $W_Z$ into a functor that preserves embeddings. To obtain a PO-dilator, we define support functions $\supp_X:W_Z(X)\to[X]^{<\omega}$ by
\begin{equation*}
\supp_X(0)=\emptyset\quad\text{and}\quad\supp_X((z,x))=\{x\}.
\end{equation*}
Naturality is satisfied in view of
\begin{multline*}
\supp_Y\circ W_Z(f)((z,x))=\supp_Y((z,f(x)))=\{f(x)\}=\\
=[f]^{<\omega}(\{x\})=[f]^{<\omega}\circ\supp_X((z,x)).
\end{multline*}
To verify the support condition we consider an embedding $f:X\to Y$ and an element $(z,y)\in W_Z(Y)$. If we have $\{y\}=\supp_Y((z,y))\subseteq\rng(f)$, then we may write $y=f(x)$ with $x\in X$. The element $(z,x)\in W_Z(X)$ witnesses
\begin{equation*}
(z,y)=(z,f(x))=W_Z(f)((z,x))\in\rng(W_Z(f)).
\end{equation*}
Let us also note that $W_Z$ is a WPO-dilator when $Z$ is a well partial order.
\end{example}

Girard~\cite{girard-pi2} has observed that LO-dilators are determined by their restrictions to the category of finite linear orders, which makes it possible to represent them in second order arithmetic. The details of such a representation have been worked out in~\cite{freund-computable}. In the following we present a similar representation for PO-dilators. 

In the linear case, each isomorphism class of finite orders contains a canonical representative of the form $n=\{0,\dots,n-1\}$ (with the usual order between natural numbers). To obtain representatives for finite partial orders, we observe that an order on a finite subset of~$\mathbb N$ can itself be coded by a natural number. Having fixed a suitable coding, we define $\poz$ as the set of all partial orders~$(a,\leq_a)$ with finite underlying set $a\subseteq\mathbb N$ that are not isomorphic to an order with smaller code. Elements of $\poz$ will be called coded partial orders. Crucially, any finite partial order $a$ (with arbitrary underlying set) is isomorphic to a unique order~$|a|\in\poz$. To ensure that our presentation is compatible with previous work on LO-dilators, we assume $n=|n|\in\poz$ for the linear orders $n=\{0,\dots,n-1\}$.

We also write $\poz$ for the category of coded partial orders and quasi embeddings. It will be important that $a\mapsto|a|$ is an equivalence between $\poz$ and the category of all finite partial orders. To witness this fact we fix an order isomorphism
\begin{equation*}
\en_a:|a|\xrightarrow{\cong} a
\end{equation*}
for each finite partial order~$a$. In the context of second order arithmetic this does not require choice, since we can pick the finite isomorphism with the smallest code. Given a function $f:a\to b$ between finite partial orders, we can define $|f|:|a|\to|b|$ as the unique function with
\begin{equation*}
\en_b\circ|f|=f\circ\en_a.
\end{equation*}
Note that $|f|$ is a (quasi) embedding whenever the same holds for~$f$. For $g:b\to c$ we have $g\circ f\circ\en_a=g\circ\en_b\circ|f|=\en_c\circ|g|\circ|f|$, so that uniqueness yields $|g|\circ|f|=|g\circ f|$. Similarly, we see that $|\id_a|$ is the identity on~$|a|$ if $\id_a$ is the identity on~$a$. In the case where~$a$ is linear, the function $\en_a:|a|=\{0,\dots,|a|-1\}\to a$ is determined as the unique increasing enumeration. If $f:a\to b$ is an embedding between linear orders, then $|f|:|a|\to|b|$ is the same function as in~\cite{freund-computable}.

Now consider a functor $W:\poz\to\po$ such that the orders~$W(a)$ for $a\in\poz$ are countable. Up to natural equivalence, we may assume that the underlying set of each order~$W(a)$ is a subset of $\mathbb N$. Coding finite structures by natural numbers, we can then represent~$W$ by the sets
\begin{align*}
W^0&=\{\langle a,\sigma,\tau\rangle\,|\,a\in\poz\text{ and }\sigma,\tau\in W(a)\text{ and }\sigma\leq_{W(a)}\tau\}\subseteq\mathbb N,\\
W^1&=\{\langle f,\sigma,\tau\rangle\,|\,f\text{ is a morphism in $\poz$ and }W(f)(\sigma)=\tau\}\subseteq\mathbb N.
\end{align*}
When we work in the base theory~$\rca_0$ of reverse mathematics, a functor~$W$ is indeed given by sets~$W^0,W^1\subseteq\mathbb N$. An expression such as $\sigma\in W(a)$ must then be read as an abbreviation for the $\Delta^0_1$-formula $\langle a,\sigma,\sigma\rangle\in W^0$. Similarly, a set
\begin{equation*}
\supp=\{\langle a,\sigma,b\rangle\,|\,\supp_a(\sigma)=b\}\subseteq\mathbb N
\end{equation*}
can encode a natural transformation $\supp:W\Rightarrow[\cdot]^{<\omega}$. When the following definition is invoked within $\rca_0$, it is assumed that we are concerned with underlying sets $W(a)\subseteq\mathbb N$. From the viewpoint of a different base theory (e.\,g.~set theory) one may also wish to consider the case where $W(a)$ is uncountable.

\begin{definition}[$\rca_0$]\label{def:coded-po-dilator}
A coded PO-dilator consists of
\begin{enumerate}[label=(\roman*)]
\item a functor $W:\poz\to\po$ that preserves embeddings and
\item a natural transformation $\supp:W\Rightarrow[\cdot]^{<\omega}$ such that the support condition
\begin{equation*}
\rng(W(f))=\{\sigma\in W(b)\,|\,\supp_b(\sigma)\subseteq\rng(f)\}.
\end{equation*}
holds whenever $f:a\to b$ is an embedding between coded partial orders.
\end{enumerate}
\end{definition}

In order to define coded WPO-dilators, we must explain how coded PO-dilators can be extended beyond finite orders (as it is not enough to demand that $W(a)$ is a wpo for every finite wpo~$a$). First, we want to show how the coded PO-dilators relate to the class-sized PO-dilators of Definition~\ref{def:po-dilator}. One direction is immediate:

\begin{lemma}\label{lem:class-to-coded}
If $W$ is a class-sized PO-dilator, then its obvious restriction $W\!\restriction\!\poz$ is a coded PO-dilator.
\end{lemma}

The lemma cannot be stated in~$\rca_0$, where the general notion of class-sized dilator is not available. Nevertheless, we can use concrete instances of the result:

\begin{example}\label{ex:coded-po-dilator}
The theory $\rca_0$ recognizes that the computable transformations $W_Z$ from Example~\ref{ex:po-dilator-aca} satisfy the defining properties of class-sized PO-dilators. It also shows that each restriction~$W_Z\!\restriction\!\poz$ exists as a set, as the category~$\poz$ is computable. Further\-more, $\rca_0$ recognizes that ~$W_Z\!\restriction\!\poz$ is a coded PO-dilator.
\end{example}

Our next goal is to extend a coded PO-dilator~$W$ into a class-sized PO-dilator~$\overline W$. The following notion, which Girard~\cite{girard-intro} has considered in the linear case, will be fundamental. The definition makes sense for coded and class-sized PO-dilators.

\begin{definition}[$\rca_0$]
The trace $\trace(W)$ of a PO-dilator~$W$ consists of all pairs $(a,\sigma)$ of a coded partial order~$a\in\poz$ and an element $\sigma\in W(a)$ that satisfy the minimality condition $\supp_a(\sigma)=a$.
\end{definition}

Intuitively speaking, the minimality condition expresses that $\sigma$ depends on all elements of~$a$. As we shall see, this ensures that certain representations are unique. The following observation is required for the definition of $\overline W(f)$ below.

\begin{lemma}[$\rca_0$]\label{lem:trace-preserved}
Consider a coded PO-dilator $W$. If $g:a\to b$ is a surjective quasi embedding between finite partial orders, then we have
\begin{equation*}
(|a|,\sigma)\in\trace(W)\quad\Rightarrow\quad (|b|,W(|g|)(\sigma))\in\trace(W).
\end{equation*}
\end{lemma}
\begin{proof}
Note that we can form $W(|g|)$, since $|g|:|a|\to|b|$ is a morphism in~$\poz$. The latter is caracterized by the equality $\en_b\circ |g|=g\circ\en_a$, where $\en_a:|a|\to a$ and $\en_b:|b|\to b$ are the isomorphisms fixed above. Now the minimality condition $\supp_{|a|}(\sigma)=|a|$ from the assumption $(|a|,\sigma)\in\trace(W)$ implies
\begin{multline*}
[\en_b]^{<\omega}\circ\supp_{|b|}(W(|g|)(\sigma))=[\en_b\circ |g|]^{<\omega}\circ\supp_{|a|}(\sigma)=\\
=[g\circ\en_a]^{<\omega}(|a|)=[g]^{<\omega}(a)=b.
\end{multline*}
This yields $\supp_{|b|}(W(|g|)(\sigma))=|b|$, as required for $(|b|,W(|g|)(\sigma))\in\trace(W)$.
\end{proof}

In order to define the extension of a coded PO-dilator, we fix some notation: When we have $a\subseteq X$, we write $\iota_a^X:a\hookrightarrow X$ for the inclusion. If $X$ is a partial order, then we consider~$a$ as a suborder, so that $\iota_a^X$ is an embedding. Given a quasi embedding $f:X\to Y$ and a finite suborder~$a\subseteq X$, we write $f\!\restriction\!a:a\to[f]^{<\omega}(a)$ for the restriction of~$f$. Here the codomain $[f]^{<\omega}(a)$ is considered as a suborder of~$Y$, so that $f\!\restriction\!a$ is a surjective quasi embedding.

\begin{definition}[$\rca_0$]\label{def:extend-po-dilator}
Let $W$ be a coded PO-dilator. For each partial order~$X$ we define a set $\overline W(X)$ and a binary relation $\leq_{\overline W(X)}$ by
\begin{gather*}
\overline W(X)=\{(a,\sigma)\,|\,a\in[X]^{<\omega}\text{ and }(|a|,\sigma)\in\trace(W)\},\\
(a,\sigma)\leq_{\overline W(X)}(b,\tau)\,\Leftrightarrow\, W(|\iota_a^{a\cup b}|)(\sigma)\leq_{W(|a\cup b|)} W(|\iota_b^{a\cup b}|)(\tau),
\end{gather*}
where $a\cup b$ is considered as a suborder of~$X$. Given a quasi embedding $f:X\to Y$, we define a function $\overline W(f):\overline W(X)\to\overline W(Y)$ by setting
\begin{equation*}
\overline W(f)((a,\sigma))=([f]^{<\omega}(a),W(|f\!\restriction\!a|)(\sigma)).
\end{equation*}
To define a family of functions $\overline\supp_X:\overline W(X)\to[X]^{<\omega}$ we put
\begin{equation*}
\overline\supp_X((a,\sigma))=a
\end{equation*}
for each partial order~$X$.
\end{definition}

Note that the following result can be stated in $\rca_0$, since it only involves class-sized dilators that are explicitly constructed from coded ones.

\begin{theorem}[$\rca_0$]\label{thm:extend-po-dilator}
If $W$ is a coded PO-dilator, then its extension $\overline W$ is a class-sized PO-dilator.
\end{theorem}
\begin{proof}
We begin by showing that $\overline W(X)$ is a partial order for any partial order~$X$. Reflexivity is readily established. In order to prove antisymmetry we must show that $W(|\iota_a^{a\cup b}|)(\sigma)=W(|\iota_b^{a\cup b}|)(\tau)$ implies $(a,\sigma)=(b,\tau)$. The minimality condition $\supp_{|a|}(\sigma)=|a|$ that is provided by $(|a|,\sigma)\in\trace(W)$ allows to recover $a$ as
\begin{multline*}
[\en_{a\cup b}]^{<\omega}\circ\supp_{|a\cup b|}(W(|\iota_a^{a\cup b}|)(\sigma))=[\en_{a\cup b}\circ|\iota_a^{a\cup b}|]^{<\omega}\circ\supp_{|a|}(\sigma)=\\ =[\iota_a^{a\cup b}\circ\en_a]^{<\omega}(|a|)=a.
\end{multline*}
As $b$ can be recovered in the same way, we see that $W(|\iota_a^{a\cup b}|)(\sigma)=W(|\iota_b^{a\cup b}|)(\tau)$ implies $a=b$. It follows that $W(|\iota_a^{a\cup b}|)$ and $W(|\iota_b^{a\cup b}|)$ are the same quasi embedding. Since quasi embeddings are injective, we can conclude $\sigma=\tau$ as well. To establish transitivity we assume $(a,\sigma)\leq_{\overline W(X)}(b,\tau)$ and $(b,\tau)\leq_{\overline W(X)}(c,\rho)$, or equivalently
\begin{equation*}
W(|\iota_a^{a\cup b}|)(\sigma)\leq_{W(|a\cup b|)} W(|\iota_b^{a\cup b}|)(\tau)\quad\text{and}\quad W(|\iota_b^{b\cup c}|)(\tau)\leq_{W(|b\cup c|)} W(|\iota_c^{b\cup c}|)(\rho).
\end{equation*}
Crucially, the condition that PO-dilators preserve embeddings allows us to deduce
\begin{multline*}
W(|\iota_a^{a\cup b\cup c}|)(\sigma)=W(|\iota_{a\cup b}^{a\cup b\cup c}|)\circ W(|\iota_a^{a\cup b}|)(\sigma)\leq_{W(|a\cup b\cup c|)}\\ \leq_{W(|a\cup b\cup c|)} W(|\iota_{a\cup b}^{a\cup b\cup c}|)\circ W(|\iota_b^{a\cup b}|)(\tau)=W(|\iota_b^{a\cup b\cup c}|)(\tau).
\end{multline*}
In the same way we get $W(|\iota_b^{a\cup b\cup c}|)(\tau)\leq_{W(|a\cup b\cup c|)} W(|\iota_c^{a\cup b\cup c}|)(\rho)$. We can then use transitivity in $W(|a\cup b\cup c|)$ to infer
\begin{multline*}
W(|\iota_{a\cup c}^{a\cup b\cup c}|)\circ W(|\iota_a^{a\cup c}|)(\sigma)=W(|\iota_a^{a\cup b\cup c}|)(\sigma)\leq_{W(|a\cup b\cup c|)}\\
\leq_{W(|a\cup b\cup c|)} W(|\iota_c^{a\cup b\cup c}|)(\rho)=W(|\iota_{a\cup c}^{a\cup b\cup c}|)\circ W(|\iota_c^{a\cup c}|)(\rho).
\end{multline*}
Since $W(|\iota_{a\cup c}^{a\cup b\cup c}|)$ is a quasi embedding, we get $W(|\iota_a^{a\cup c}|)(\sigma)\leq_{W(|a\cup c|)} W(|\iota_c^{a\cup c}|)(\rho)$. This amounts to $(a,\sigma)\leq_{\overline W(X)} (c,\rho)$, as required for transitivity. Next, we show that $\overline W(f)$ is a quasi embedding for any quasi embedding $f:X\to Y$. For this purpose we consider an inequality $\overline W(f)((a,\sigma))\leq_{\overline W(Y)}\overline W(f)((b,\tau))$, which amounts to
\begin{equation*}
W\left(\left|\iota_{[f]^{<\omega}(a)}^{[f]^{<\omega}(a\cup b)}\right|\circ \left|f\!\restriction\!a\right|\right)(\sigma)\leq_{W(|[f]^{<\omega}(a\cup b)|)}W\left(\left|\iota_{[f]^{<\omega}(b)}^{[f]^{<\omega}(a\cup b)}\right|\circ \left|f\!\restriction\!b\right|\right)(\tau).
\end{equation*}
In view of $\iota_{[f]^{<\omega}(a)}^{[f]^{<\omega}(a\cup b)}\circ (f\!\restriction\!a)=f\!\restriction\!(a\cup b)\circ\iota_a^{a\cup b}$ we can infer
\begin{equation*}
W(|f\!\restriction\!(a\cup b)|)\circ W(|\iota_a^{a\cup b}|)(\sigma)\leq_{W(|[f]^{<\omega}(a\cup b)|)} W(|f\!\restriction\!(a\cup b)|)\circ W(|\iota_b^{a\cup b}|)(\tau).
\end{equation*}
As $W(|f\!\restriction\!(a\cup b)|)$ is a quasi embedding, we get $W(|\iota_a^{a\cup b}|)(\sigma)\leq_{W(|a\cup b|)}W(|\iota_b^{a\cup b}|)(\tau)$. This amounts to $(a,\sigma)\leq_{\overline W(X)}(b,\tau)$, as required to show that $\overline W(f)$ is a quasi embedding. If $f$ is an embedding, then the argument can be read in reverse, which reveals that $\overline W(f)$ is an embedding as well. To see that $\overline W$ is a functor it suffices to recall that $[\cdot]^{<\omega}$ and $|\cdot|$ are functorial. The naturality of $\overline\supp$ is evident from the definition. It remains to verify the support condition from clause~(ii) of Definition~\ref{def:po-dilator}. For this purpose we consider an embedding $f:X\to Y$ and an element $(a,\sigma)\in\overline W(Y)$ with $a=\overline\supp_Y((a,\sigma))\subseteq\rng(f)$. Due to the latter, we can write $a=[f]^{<\omega}(a_0)$ with $a_0\in[X]^{<\omega}$. Since $f\!\restriction\!a_0:a_0\to a$ is surjective and we have $\en_a\circ|f\!\restriction\!a_0|=(f\!\restriction\!a_0)\circ\en_{a_0}$, we get $\rng(|f\!\restriction\!a_0|)=|a|=\supp_{|a|}(\sigma)$. Hence the support condition from clause~(ii) of Definition~\ref{def:coded-po-dilator} allows us to write
\begin{equation*}
\sigma=W(|f\!\restriction\!a_0|)(\sigma_0)\quad\text{with}\quad \sigma_0\in W(|a_0|).
\end{equation*}
Let us observe that we have
\begin{equation*}
[|f\!\restriction\!a_0|]^{<\omega}\circ\supp_{|a_0|}(\sigma_0)=\supp_{|a|}(W(|f\!\restriction\!a_0|)(\sigma_0))=\supp_{|a|}(\sigma)=|a|.
\end{equation*}
Since the quasi embedding $|f\!\restriction\!a_0|$ is injective, this implies $\supp_{|a_0|}(\sigma_0)=|a_0|$. We thus have $(|a_0|,\sigma_0)\in\trace(W)$ and hence $(a_0,\sigma_0)\in\overline W(X)$. By construction we get
\begin{equation*}
(a,\sigma)=([f]^{<\omega}(a_0),W(|f\!\restriction\!a_0|)(\sigma_0))=\overline W(f)((a_0,\sigma_0))\in\rng(\overline W(f)),
\end{equation*}
as required by the support condition.
\end{proof}

Given a class-sized PO-dilator~$W$ (with support $\supp^W:W\Rightarrow[\cdot]^{<\omega}$), Lemma~\ref{lem:class-to-coded} yields a coded PO-dilator~$W\!\restriction\!\poz$. By the previous theorem we get another class-sized PO-dilator $\overline{W\!\restriction\!\poz}$ (with support $\overline\supp^{W\restriction\poz}:\overline{W\!\restriction\!\poz}\Rightarrow[\cdot]^{<\omega}$). The following result shows that we have indeed reconstructed the original PO-dilator~$W$.

\begin{theorem}\label{thm:reconstruct-class-dilator}
Given any class-sized PO-dilator~$W$, one can construct a natural isomorphism $\eta:\overline{W\!\restriction\!\poz}\Rightarrow W$ such that we have $\overline\supp^{W\restriction\poz}=\supp^W\circ\eta$.
\end{theorem}
In the following proof, the condition $\overline\supp^{W\restriction\poz}=\supp^W\circ\eta$ is verified explicitly. More generally, one can show that any natural transformation between PO-dilators preserves supports. For the linear case this has been established by Girard~\cite{girard-pi2}. The detailed proof in~\cite[Lemma~2.17]{freund-rathjen_derivatives} is readily adapted to the partial case.
\begin{proof}
To simplify the notation we will write $\overline W$ rather than $\overline{W\!\restriction\!\poz}$. Recall that elements of $\overline W(X)$ are of the form $(a,\sigma)$ with $a\in[X]^{<\omega}$ and $\sigma\in W(|a|)$. For any partial order~$X$ we can thus define a function $\eta_X:\overline W(X)\to W(X)$ by setting
\begin{equation*}
\eta_X((a,\sigma))=W(\iota_a^X\circ\en_a)(\sigma).
\end{equation*} 
To see that this yields an embedding, we recall that $(a,\sigma)\leq_{\overline W(X)}(b,\sigma)$ amounts to
\begin{equation*}
W(|\iota_a^{a\cup b}|)(\sigma)\leq_{W(|a\cup b|)} W(|\iota_b^{a\cup b}|)(\tau).
\end{equation*}
Since PO-dilators preserve embeddings, this is equivalent to
\begin{equation*}
W(\iota_{a\cup b}^X\circ\en_{a\cup b}\circ|\iota_a^{a\cup b}|)(\sigma)\leq_{W(X)}W(\iota_{a\cup b}^X\circ\en_{a\cup b}\circ|\iota_b^{a\cup b}|)(\tau).
\end{equation*}
In view of $\iota_{a\cup b}^X\circ\en_{a\cup b}\circ|\iota_a^{a\cup b}|=\iota_{a\cup b}^X\circ\iota_a^{a\cup b}\circ\en_a=\iota_a^X\circ\en_a$ we can conclude that the inequality $(a,\sigma)\leq_{\overline W(X)}(b,\sigma)$ is equivalent to
\begin{equation*}
\eta_X((a,\sigma))=W(\iota_a^X\circ\en_a)(\sigma)\leq_{W(X)}W(\iota_b^X\circ\en_b)(\tau)=\eta_X((b,\tau)),
\end{equation*}
as desired. To conclude that $\eta_X$ is an isomorphism it remains to prove surjectivity. Given an arbitrary element $\sigma\in W(X)$, we set $a:=\supp_X(\sigma)$. Since $\iota_a^X\circ\en_a$ has range $a$, the support condition yields $\sigma=W(\iota_a^X\circ\en_a)(\sigma_0)$ for some $\sigma_0\in W(|a|)$. Due to the naturality of supports, we have
\begin{equation*}
[\iota_a^X\circ\en_a]^{<\omega}\circ\supp_{|a|}(\sigma_0)=\supp_X(W(\iota_a^X\circ\en_a)(\sigma_0))=\supp_X(\sigma)=a.
\end{equation*}
As $\iota_a^X\circ\en_a$ is injective, this implies $\supp_{|a|}(\sigma_0)=|a|$. Thus we get $(|a|,\sigma_0)\in\trace(W)$ and hence $(a,\sigma_0)\in\overline W(X)$. By construction we have $\eta_X((a,\sigma_0))=\sigma$, as needed. Let us now show that $\eta$ is natural. For a quasi embedding $f:X\to Y$ and any finite suborder $a\subseteq X$ we have
\begin{equation*}
\iota_{[f]^{<\omega}(a)}^Y\circ\en_{[f]^{<\omega}(a)}\circ|f\!\restriction\!a|=\iota_{[f]^{<\omega}(a)}^Y\circ(f\!\restriction\!a)\circ\en_a=f\circ\iota_a^X\circ\en_a.
\end{equation*}
Given an element $(a,\sigma)\in\overline W(X)$, we now obtain
\begin{multline*}
\eta_Y\circ\overline W(f)((a,\sigma))=\eta_Y(([f]^{<\omega}(a), W(|f\!\restriction\!a|)(\sigma)))=\\
=W(\iota_{[f]^{<\omega}(a)}^Y\circ\en_{[f]^{<\omega}(a)}\circ|f\!\restriction\!a|)(\sigma)=W(f\circ\iota_a^X\circ\en_a)(\sigma)=W(f)\circ\eta_X((a,\sigma)).
\end{multline*}
Finally, we verify that $\eta$ preserves supports. Recall that any element $(a,\sigma)\in\overline W(X)$ satisfies the minimality condition~$\supp^W_{|a|}(\sigma)=|a|$. We can deduce
\begin{multline*}
\supp^W_X\circ\eta_X((a,\sigma))=\supp^W_X(W(\iota_a^X\circ\en_a)(\sigma))=[\iota_a^X\circ\en_a]^{<\omega}\circ\supp^W_{|a|}(\sigma)=\\
=[\iota_a^X\circ\en_a]^{<\omega}(|a|)=a=\overline\supp^{W\restriction\poz}_X((a,\sigma)),
\end{multline*}
as the theorem claims.
\end{proof}

In the following we revert to questions of well partial orderedness.

\begin{definition}[$\rca_0$]\label{def:coded-wpo-dilator}
A coded PO-dilator~$W$ is called a coded WPO-dilator if $\overline W(X)$ is a well partial order (wpo) for any wpo~$X$.
\end{definition}

If the previous definition is evaluated in~$\rca_0$, then we can only consider countable orders~$X$ (with underlying set~$X\subseteq\mathbb N$). The following result shows that this restriction is harmless: the notion of coded WPO-dilator does not change its meaning when we pass to a more expressive setting. We point out that Girard~\cite{girard-pi2} has established the same result for the linear case.

\begin{proposition}\label{prop:dilator-countable-wpo}
Consider a class-sized PO-dilator~$W$. If $W(X)$ is a wpo for every countable wpo~$X$, then the same holds when~$X$ is an uncountable wpo.
\end{proposition}
\begin{proof}
Consider an arbitrary wpo~$X$ and an infinite sequence $\sigma_0,\sigma_1,\dots$ in~$W(X)$. Since all supports are finite, the suborder
\begin{equation*}
Z:=\bigcup\{\supp_X(\sigma_i)\,|\,i\in\mathbb N\}\subseteq X
\end{equation*}
is a countable wpo. For each $i\in\mathbb N$ we have $\supp_X(\sigma_i)\subseteq Z=\rng(\iota_Z^X)$, so that the support condition yields $\sigma_i=W(\iota_Z^X)(\tau_i)$ for some $\tau_i\in W(Z)$. Given that~$W(Z)$ is a wpo, we find indizes $i<j$ with $\tau_i\leq_{W(Z)}\tau_j$. As PO-dilators preserve embeddings, we can conclude $\sigma_i\leq_{W(X)}\sigma_j$, as needed to show that $W(X)$ is a wpo.
\end{proof}

By combining previous results, we obtain the following:

\begin{corollary}
If $W$ is a class-sized WPO-dilator, then $W\!\restriction\!\poz$ is a coded WPO-dilator. If $W$ is a coded WPO-dilator, then $\overline W$ is a class-sized WPO-dilator.
\end{corollary}
\begin{proof}
First assume that~$W$ is a class-sized WPO-dilator. It is straightforward to see that $W\!\restriction\!\poz$ is a coded PO-dilator (cf.~Lemma~\ref{lem:class-to-coded}). Given a wpo~$X$, the isomorphism from Theorem~\ref{thm:reconstruct-class-dilator} ensures that $\overline{W\!\restriction\!\poz}(X)\cong W(X)$ is a wpo as~well. According to Definition~\ref{def:coded-wpo-dilator}, this means that $W\!\restriction\!\poz$ is a coded WPO-dilator. For the second part of the corollary, we assume that $W$ is a coded WPO-dilator. In Theorem~\ref{thm:extend-po-dilator} we have shown that $\overline W$ is a class-sized PO-dilator. To conclude that $\overline W$ is a class-sized WPO-dilator, it suffices to known that $\overline W(X)$ is a wpo for any wpo~$X$. This is immediate by Definition~\ref{def:coded-wpo-dilator}.
\end{proof}

As a general statement about all class-sized PO-dilators, Theorem~\ref{thm:reconstruct-class-dilator} cannot be formalized in $\rca_0$. However, concrete instances of the theorem are available (and very useful) in our base theory:

\begin{example}\label{ex:coded-class-iso}
Recall the class-sized PO-dilators $W_Z(X)=1+Z\times X$ from Example~\ref{ex:po-dilator-aca}. In Example~\ref{ex:coded-po-dilator} we have seen that the coded PO-dilators $W_Z\!\restriction\!\poz$ are available in $\rca_0$. For each pair of partial orders $Z$ and~$X$, the isomorphism
\begin{equation*}
\eta_X:\overline{W_Z\!\restriction\!\poz}(X)\xrightarrow{\cong} W_Z(X)
\end{equation*}
from Theorem~\ref{thm:reconstruct-class-dilator} can be constructed in $\rca_0$ as well. To conclude that $W_Z\!\restriction\!\poz$ is a coded WPO-dilator, it is thus enough to show that $W_Z(X)$ is a well partial order whenever the same holds for $Z$ and~$X$. The latter can be proved in $\rca_0+\cac$ but not in~$\mathbf{WKL}_0\supseteq\rca_0$, due to a result of Cholak, Marcone and Solomon~\cite{cholak-RM-wpo}.
\end{example}

We have seen that class-sized PO-dilators are equivalent to coded PO-dilators, and that important parts of the equivalence can be established in~$\rca_0$ (at least for PO-dilators with countable trace). In the sequel, the specifications ``class-sized" and ``coded" will sometimes be left implicit.

\section{A uniform Kruskal theorem}\label{sect:generalized-Kruskal}

In this section we construct a partial order $\T W$ relative to a given PO-dilator~$W$, which needs to satisfy a certain normality condition. As explained in the introduction, one can view $\T W$ as a fixed point of~$W$ (see Theorem~\ref{thm:initial-fixed-point} below for a precise statement). Higman's order between finite sequences and various versions of the order from Kruskal's theorem are all of the form $\T W$ for a suitable WPO-dilator~$W$. Using the minimal bad sequence argument, we will show that $\T W$ is a well partial order whenever $W$ is a WPO-dilator.

To state the aforementioned normality condition, we need some terminology. Recall that $[X]^{<\omega}$ denotes the set of finite subsets of~$X$. If $\leq_X$ is a partial order on~$X$, then a quasi order~$\leqf_X$ on $[X]^{<\omega}$ can be given by
\begin{equation*}
a\leqf_X b\quad\Leftrightarrow\quad\text{for any $x\in a$ there is an $x'\in b$ with $x\leq_X x'$}.
\end{equation*}
In the case of singletons we will write $a\leqf_X x'$ and $x\leqf_X b$ rather than $a\leqf_X\{x'\}$ and $\{x\}\leqf_X b$, respectively.

\begin{definition}[$\rca_0$]\label{def:PO-dilator-normal}
A coded PO-dilator $(W,\supp)$ is called normal if we have
\begin{equation*}
\sigma\leq_{W(a)}\tau\quad\Rightarrow\quad\supp_a(\sigma)\leqf_a\supp_a(\tau)
\end{equation*}
for any $a\in\poz$ and all elements $\sigma,\tau\in W(a)$.
\end{definition}

In the previous section we have seen that any coded PO-dilator~$W$ extends into a class-sized PO-dilator~$\overline W$. Let us recall that elements of $\overline W(X)$ are of the form $(a,\sigma)$ with $a\in[X]^{<\omega}$ and $\sigma\in W(|a|)$. The support functions associated with $\overline W$ are given by $\overline\supp_X((a,\sigma))=a$. The following result shows that the normality condition extends beyond the finite orders in $\poz$. In view of Theorem~\ref{thm:reconstruct-class-dilator}, the result applies to all class-sized PO-dilators.

\begin{lemma}[$\rca_0$]\label{lem:normal-coded-class}
If $W$ is a (coded) normal PO-dilator, then we have
\begin{equation*}
(a,\sigma)\leq_{\overline W(X)}(b,\tau)\quad\Rightarrow\quad a\leqf_X b
\end{equation*}
for any partial order~$X$ and all elements $(a,\sigma),(b,\tau)\in\overline W(X)$.
\end{lemma}
\begin{proof}
Assume that we have $(a,\sigma)\leq_{\overline W(X)}(b,\tau)$, which amounts to
\begin{equation*}
W(|\iota_a^{a\cup b}|)(\sigma)\leq_{W(|a\cup b|)} W(|\iota_b^{a\cup b}|)(\tau).
\end{equation*}
Since~$W$ is normal, this inequality implies
\begin{equation*}
\supp_{|a\cup b|}(W(|\iota_a^{a\cup b}|)(\sigma))\leqf_{|a\cup b|}\supp_{|a\cup b|}(W(|\iota_b^{a\cup b}|)(\tau)).
\end{equation*}
Recall that any $(a,\sigma)\in\overline W(X)$ satisfies $(|a|,\sigma)\in\trace(W)$ and hence $\supp_{|a|}(\sigma)=|a|$. We thus have
\begin{equation*}
\supp_{|a\cup b|}(W(|\iota_a^{a\cup b}|)(\sigma))=[|\iota_a^{a\cup b}|]^{<\omega}\circ\supp_{|a|}(\sigma)=[|\iota_a^{a\cup b}|]^{<\omega}(|a|).
\end{equation*}
Hence the above amounts to
\begin{equation*}
[|\iota_a^{a\cup b}|]^{<\omega}(|a|)\leqf_{|a\cup b|} [|\iota_b^{a\cup b}|]^{<\omega}(|b|).
\end{equation*}
Invoking the isomorphisms $\en_a:|a|\to a$ and $\en_{a\cup b}:|a\cup b|\to a\cup b$, we have
\begin{equation*}
\iota_a^X\circ\en_a=\iota_{a\cup b}^X\circ\iota_a^{a\cup b}\circ\en_a=\iota_{a\cup b}^X\circ\en_{a\cup b}\circ|\iota_a^{a\cup b}|.
\end{equation*}
Since $\iota_{a\cup b}^X\circ\en_{a\cup b}$ is an embedding, we can conclude
\begin{multline*}
a=[\iota_a^X\circ\en_a]^{<\omega}(|a|)=[\iota_{a\cup b}^X\circ\en_{a\cup b}]^{<\omega}\circ[|\iota_a^{a\cup b}|]^{<\omega}(|a|)\leqf_X\\
\leqf_X [\iota_{a\cup b}^X\circ\en_{a\cup b}]^{<\omega}\circ[|\iota_b^{a\cup b}|]^{<\omega}(|b|)=[\iota_b^X\circ\en_b]^{<\omega}(|b|)=b,
\end{multline*}
as desired.
\end{proof}

Labelled structures are often ordered by embeddings that map each node to a node with bigger label. This condition on the labels ensures that the resulting PO-dilator is normal.

\begin{example}\label{ex:W_Z-normal}
The PO-dilators $W_Z(X)=1+Z\times X$ from Example~\ref{ex:po-dilator-aca} are normal. To see that this is the case, we consider an inequality
\begin{equation*}
(z,x)\leq_{W_Z(X)} (z',x').
\end{equation*}
We then have $z\leq_Z z'$ and $x\leq_X x'$. The latter yields
\begin{equation*}
\supp_X((z,x))=\{x\}\leqf_X \{x'\}=\supp_X((z',x')),
\end{equation*}
as required.
\end{example}

The term ``normal" is motivated by the linear case, where normal WO-dilators induce normal functions on the ordinals (due to Aczel~\cite{aczel-normal-functors}). As mentioned in the introduction, Freund~\cite{freund-equivalence,freund-categorical,freund-computable} has shown that $\Pi^1_1$-comprehension is equivalent to the principle that $\vartheta(D)$ is well-founded for any WO-dilator~$D$. Interestingly, the latter becomes much weaker when we require normality: the principle that normal WO-dilators have well founded fixed points does only lead up to (bar) induction for $\Pi^1_1$-formulas, as shown by Freund and Rathjen~\cite{freund-rathjen_derivatives,freund-single-fixed-points,freund-ordinal-exponentiation}. In contrast, the present paper shows that $\Pi^1_1$-comprehension is equivalent to the principle that $\T W$ is a wpo for any normal WPO-dilator~$W$. This reveals that the normality condition behaves rather differently in the partial case.

Weiermann~\cite{weiermann09} has previously described the construction of a partial order~$\T W$ relative to a transformation~$W$ of partial orders. So far, the general construction has been a successful heuristic principle: it has led to the analysis of several well partial orders from algebra and combinatorics by van der Meeren, Pelupessy, Rathjen and Weiermann~\cite{pelupessy-weiermann,MRW-Veblen,MRW-Bachmann}. Based on the notion of normal PO-dilator, we can now make the general construction official. At two points in the following definition, we require that $\leq_{\T W}$ is a partial order on certain subsets of $\T W$. Eventually, this requirement will turn out to be redundant, as the entire set $\T W$ is partially ordered by $\leq_{\T W}$. A detailed justification of the following recursion can be found below.

\begin{definition}[$\rca_0$]\label{def:T(W)}
Consider a normal PO-dilator~$W$. We define a set $\T W$ of terms and a binary relation $\leq_{\T W}$ on this set by simultaneous recursion:
\begin{itemize}
\item Given a finite set $a\subseteq\T W$ that is partially ordered by $\leq_{\T W}$, we add a term $\circ(a,\sigma)\in\T W$ for each $\sigma\in W(|a|)$ with $(|a|,\sigma)\in\trace(W)$.
\item We have $\circ(a,\sigma)\leq_{\T W}\circ(b,\tau)$ if, and only if, one of the following holds:
\begin{enumerate}[label=(\roman*)]
\item We have $\circ(a,\sigma)\leq_{\T W}t$ for some $t\in b$.
\item The set $a\cup b$ is partially ordered by $\leq_{\T W}$ and we have
\begin{equation*}
W(|\iota_a^{a\cup b}|)(\sigma)\leq_{W(|a\cup b|)}W(|\iota_b^{a\cup b}|)(\tau).
\end{equation*}
\end{enumerate}
\end{itemize}
\end{definition}

Let us point out that $\T W$ is non-empty if the same holds for $W(\emptyset)$. In several natural examples this is the case: if $W(X)$ consists of the finite sequences with entries in~$X$, then $W(\emptyset)$ contains the empty sequence~$\langle\rangle$. In cases where $W(\emptyset)$ is empty, one can work with $X\mapsto 1+W(X)$ rather than~$W$ (cf.~Example~\ref{ex:po-dilator-aca}).

To justify the previous definition in more detail, one can proceed as follows: First generate a set $\T_0W\supseteq\T W$ by including all terms $\circ(a,\sigma)$ for finite $a\subseteq\T_0W$, where $a$ is not assumed to be ordered and $\sigma$ may be the second component of any pair in~$\trace(W)$. Let us write $\ulcorner s\urcorner$ for the G\"odel number of a term~$s\in\T_0W$ (note that $\ulcorner s\urcorner$ and $s$ coincide if the construction is already arithmetized). Now we define a length function $l:\T_0W\to\mathbb N$ by the recursive clause
\begin{equation*}
l(\circ(a,\sigma))=\max\left\{\ulcorner\circ(a,\sigma)\urcorner,1+\textstyle\sum_{s\in a}2\cdot l(s)\right\}.
\end{equation*}
The G\"odel numbers have been included to ensure that quantifier occurrences of the form $\forall_{s\in\T_0W}(l(s)\leq n\rightarrow\dots)$ are bounded, which justifies certain induction arguments in~$\rca_0$. For $r,s,t\in\T_0W$ one can now decide $r\in\T W$ and $s\leq_{\T W}t$ by simultaneous recursion on $l(r)$ and $l(s)+l(t)$. Indeed, to decide $r=\circ(a,\sigma)\in\T W$ we first decide $s\leq_{\T W}t$ for all $s,t\in a$, which is possible in view of $l(s)+l(t)<l(r)$. If the resulting relation $\leq_{\T W}$ is a partial order on $a\subseteq\T W$, then we determine the unique~$|a|\in\poz$ with $(a,\leq_{\T W})\cong|a|$. Finally, we check $(|a|,\sigma)\in\trace(W)$.

To show that $\leq_{\T W}$ is a partial order, we also need a height $h:\T W\to\mathbb N$ with
\begin{equation*}
h(\circ( a,\sigma))=\max(\{0\}\cup\{h(s)+1\,|\,s\in a\}).
\end{equation*}
Normality has the following important consequence:

\begin{lemma}[$\rca_0$]\label{lem:order-heights}
Assume that $W$ is a normal PO-dilator. We have
\begin{equation*}
s\leq_{\T W}t\quad\Rightarrow\quad h(s)\leq h(t)
\end{equation*}
for all $s,t\in\T W$.
\end{lemma}
\begin{proof}
We establish the claim by induction on $l(s)+l(t)$. Assume that we have
\begin{equation*}
s=\circ( a,\sigma)\leq_{\T W}\circ( b,\tau)=t.
\end{equation*}
Let us first consider the case where we have $s\leq_{\T W}t'$ for some $t'\in b$. In view of $l(t')<l(t)$ we inductively get $h(s)\leq h(t')<h(t)$. Now assume that $a\cup b$ is partially ordered by $\leq_{\T W}$ and that we have
\begin{equation*}
W(|\iota_a^{a\cup b}|)(\sigma)\leq_{W(|a\cup b|)}W(|\iota_b^{a\cup b}|)(\tau).
\end{equation*}
As in the proof of Lemma~\ref{lem:normal-coded-class} (with $a\cup b$ at the place of~$X$), we can use normality to get $a\leqf_{\T W}b$. Given an arbitrary $s'\in a$, we thus find a $t'\in b$ with $s'\leq_{\T W}t'$. In view of $l(s')+l(t')<l(s)+l(t)$, the induction hypothesis yields $h(s')\leq h(t')<h(t)$. Since this holds for all $s'\in a$, we obtain
\begin{equation*}
h(s)=\max(\{0\}\cup\{h(s')+1\,|\,s'\in a\})\leq h(t),
\end{equation*}
as required.
\end{proof}

In the following proof, normality is used to show antisymmetry and transitivity.

\begin{proposition}[$\rca_0$]\label{prop:TW-order}
The relation $\leq_{\T W}$ is a partial order on $\T W$, for any normal PO-dilator~$W$.
\end{proposition}
\begin{proof}
By induction on $n$ we simultaneously show
\begin{align*}
r\leq_{\T W}r& && \text{for $l(r)\leq n$},\\
s\leq_{\T W}t\land t\leq_{\T W}s&\,\Rightarrow\, s=t && \text{for $l(s)+l(t)\leq n$},\\
r\leq_{\T W}s\land s\leq_{\T W}t&\,\Rightarrow\, r\leq_{\T W} t && \text{for $l(r)+l(s)+l(t)\leq n$}.
\end{align*}
Concerning reflexivity for $r=\circ( a,\sigma)$, the induction hypothesis ensures that $\leq_{\T W}$ is a partial order on $a$ (due to the factor $2$ in the definition of our length function). Reflexivity in the partial order $W(|a|)$ yields $W(|\iota_a^a|)(\sigma)\leq_{W(|a|)} W(|\iota_a^a|)(\sigma)$. We can thus conclude $r\leq_{\T W} r$ by clause~(ii) of Definition~\ref{def:T(W)}. Let us now establish antisymmetry for $s=\circ( a,\sigma)$ and $t=\circ( b,\tau)$. First assume that $s\leq_{\T W}t$ holds because we have $s\leq_{\T W}t'$ for some $t'\in b$. Using the previous lemma, we then obtain $h(s)\leq h(t')<h(t)$. This means that $t\leq_{\T W}s$ cannot hold, again by the previous lemma. A symmetric argument applies if we have $t\leq_{\T W}s'$ for some~$s'\in a$. It remains to consider the case where we have
\begin{equation*}
W(|\iota_a^{a\cup b}|)(\sigma)\leq_{W(|a\cup b|)}W(|\iota_b^{a\cup b}|)(\tau)\quad\text{and}\quad W(|\iota_b^{a\cup b}|)(\tau)\leq_{W(|a\cup b|)}W(|\iota_a^{a\cup b}|)(\sigma).
\end{equation*}
By antisymmetry in the partial order $W(|a\cup b|)$ we get $W(|\iota_a^{a\cup b}|)(\sigma)=W(|\iota_b^{a\cup b}|)(\tau)$. As in the proof of Theorem~\ref{thm:extend-po-dilator} we can deduce $a=b$ and $\sigma=\tau$, which yields the desired equality~$s=t$. Finally, we establish transitivity for $r=\circ( a,\sigma)$, $s=\circ( b,\tau)$ and $t=\circ( c,\rho)$. First assume that $s\leq_{\T W}t$ holds because we have $s\leq_{\T W}t'$ for some $t'\in c$. By induction hypothesis we get $r\leq_{\T W}t'$ and then $r\leq_{\T W}t$. For the rest of the argument we may assume
\begin{equation*}
W(|\iota_b^{b\cup c}|)(\tau)\leq_{W(|b\cup c|)}W(|\iota_c^{b\cup c}|)(\rho).
\end{equation*}
As in the proof of Lemma~\ref{lem:normal-coded-class} we can use normality to get $b\leqf_{\T W} c$. Now assume that $r\leq_{\T W}s$ holds because we have $r\leq_{\T W}s'$ for some $s'\in b$. Due to $b\leqf_{\T W} c$, we get a $t'\in c$ with $s'\leq_{\T W}t'$. By the induction hypothesis we obtain $r\leq_{\T W} t'$ and then $r\leq_{\T W}t$. It remains to consider the case where we have
\begin{equation*}
W(|\iota_a^{a\cup b}|)(\sigma)\leq_{W(|a\cup b|)}W(|\iota_b^{a\cup b}|)(\tau).
\end{equation*}
The induction hypothesis ensures that $\leq_{\T W}$ is a partial order on $a\cup b\cup c$. As in the proof of Theorem~\ref{thm:extend-po-dilator} we can then deduce $W(|\iota_a^{a\cup c}|)(\sigma)\leq_{W(|a\cup c|)}W(|\iota_c^{a\cup c}|)(\rho)$. This yields $r=\circ(a,\sigma)\leq_{\T W}\circ(c,\rho)=t$, as needed for transitivity.
\end{proof}

In the previous section we have chosen a unique representative~$|a|\in\poz$ from the isomorphism class of each finite partial order~$a$. The terms in $\T W$ depend on this choice, due to the condition $\sigma\in W(|a|)$ in Definition~\ref{def:T(W)}. Even for very simple examples of a normal PO-dilator~$W$, this makes it hard to give an understandable description of~$\T W$. In order to solve this problem, we now present a categorical characterization, which determines the order~$\T W$ up to isomorphism. As before, we write $\overline\supp_X:\overline W(X)\to[X]^{<\omega}$ for the support functions that come with the class-sized extension~$\overline W$ of a coded PO-dilator~$W$.

\begin{definition}[$\rca_0$]\label{def:Kruskal-fixed-point}
Let~$W$ be a normal PO-dilator. A Kruskal fixed point of~$W$ consists of a partial order~$X$ and a bijection $\kappa:\overline W(X)\to X$ with
\begin{equation*}
\kappa(\sigma)\leq_X \kappa(\tau)\quad\Leftrightarrow\quad \text{we have }\sigma\leq_{\overline W(X)}\tau\text{ or }\kappa(\sigma)\leqf_X\overline\supp_X(\tau)
\end{equation*}
for all $\sigma,\tau\in\overline W(X)$. We say that $(X,\kappa)$ is initial if any other Kruskal fixed point $(X',\kappa')$ of~$W$ admits a unique quasi embedding $f:X\to X'$ with $f\circ\kappa=\kappa'\circ\overline W(f)$.
\end{definition}

Just as other initial objects, initial Kruskal fixed points of a normal PO-dilator are unique up to isomorphism. Hence the following determines~$\T W$.

\begin{theorem}[$\rca_0$]\label{thm:initial-fixed-point}
Consider a normal PO-dilator~$W$. There is a function $\kappa:\overline W(\T W)\to\T W$ such that $(\T W,\kappa)$ is an initial Kruskal fixed point of~$W$.
\end{theorem}
\begin{proof}
Comparing Definition~\ref{def:extend-po-dilator} and Definition~\ref{def:T(W)}, we see that the conditions for $(a,\sigma)\in\overline W(\T W)$ and for $\circ(a,\sigma)\in\T W$ are almost the same. In Definition~\ref{def:T(W)} we have included the additional condition that $\leq_{\T W}$ must be a partial order on~$a$. Due to Proposition~\ref{prop:TW-order}, we now know that this condition is automatic. We can thus define a bijection $\kappa:\overline W(\T W)\to\T W$ by setting
\begin{equation*}
\kappa((a,\sigma))=\circ(a,\sigma).
\end{equation*}
To show that $(\T W,\kappa)$ is a Kruskal fixed point, we consider arbitrary elements $(a,\sigma)$ and $(b,\tau)$ of~$\overline W(\T W)$. By comparing Definition~\ref{def:T(W)} and Definition~\ref{def:Kruskal-fixed-point}, we see that it is enough to convince ourselves of
\begin{align*}
\circ(a,\sigma)\leqf_{\T W}b\quad &\Leftrightarrow\quad \kappa((a,\sigma))\leqf_{\T W}\overline\supp_X((b,\tau)),\\
W(|\iota_a^{a\cup b}|)(\sigma)\leq_{W(|a\cup b|)}W(|\iota_b^{a\cup b}|)(\tau)\quad &\Leftrightarrow\quad (a,\sigma)\leq_{\overline W(\T W)}(b,\tau).
\end{align*}
The first equivalence is true in view of $\kappa((a,\sigma))=\circ(a,\sigma)$ and $\overline\supp_X((b,\tau))=b$. The second equivalence holds by Definition~\ref{def:extend-po-dilator}. Now consider another Kruskal fixed point $(X,\kappa')$. In view of Definition~\ref{def:extend-po-dilator}, the condition $f\circ\kappa=\kappa'\circ\overline W(f)$ from Definition~\ref{def:Kruskal-fixed-point} is equivalent to
\begin{equation*}
f(\circ(a,\sigma))=\kappa'(([f]^{<\omega}(a),W(|f\!\restriction\! a|)(\sigma))).
\end{equation*}
Since $s\in a$ implies $l(s)<l(\circ(a,\sigma))$, this can be read as a recursive definition, which admits at most one solution~$f$. To complete the proof we must show that the given recursion does indeed yield a quasi embedding $f:\T W\to X$. We establish
\begin{align*}
r\in\T W\quad&\Rightarrow\quad f(r)\in X,\\
f(s)\leq_X f(t)\quad&\Rightarrow\quad s\leq_{\T W}t
\end{align*}
by simultaneous induction on $l(r)$ and $l(s)+l(t)$, respectively. Let us show the first implication for $r=\circ(a,\sigma)$. The simultaneous induction hypothesis ensures that we have $[f]^{<\omega}(a)\in[X]^{<\omega}$, and that $f\!\restriction\! a:a\to [f]^{<\omega}(a)$ is a quasi embedding. In view of $r\in\T W$ we have $(|a|,\sigma)\in\trace(W)$, so that Lemma~\ref{lem:trace-preserved} yields
\begin{equation*}
(|[f]^{<\omega}(a)|,W(|f\!\restriction\! a|)(\sigma))\in\trace(W).
\end{equation*}
Due to Definition~\ref{def:extend-po-dilator} we get $([f]^{<\omega}(a),W(|f\!\restriction\! a|)(\sigma))\in\overline W(X)$, and hence $f(r)\in X$. To establish the second implication for $s=\circ(a,\sigma)$ and $t=\circ(b,\tau)$, we assume
\begin{equation*}
f(s)=\kappa'(([f]^{<\omega}(a),W(|f\!\restriction\! a|)(\sigma)))\leq_X\kappa'(([f]^{<\omega}(b),W(|f\!\restriction\! b|)(\tau)))=f(t).
\end{equation*}
Since $(X,\kappa')$ is a Kruskal fixed point, one of the following two cases must apply: First assume that we have
\begin{equation*}
([f]^{<\omega}(a),W(|f\!\restriction\! a|)(\sigma))\leq_{\overline W(X)} ([f]^{<\omega}(b),W(|f\!\restriction\! b|)(\tau)).
\end{equation*}
The induction hypothesis ensures that $f\!\restriction\!(a\cup b):a\cup b\to[f]^{<\omega}(a\cup b)\subseteq X$ is a quasi embedding. Arguing as in the proof of Theorem~\ref{thm:extend-po-dilator}, we can then deduce~$(a,\sigma)\leq_{\overline W(\T W)}(b,\tau)$. Since $(\T W,\kappa)$ is a Kruskal fixed point, we obtain
\begin{equation*}
s=\kappa((a,\sigma))\leq_{\T W}\kappa((b,\tau))=t.
\end{equation*}
Now assume that $f(s)\leq_X f(t)$ holds because of
\begin{equation*}
f(s)\leqf_X\overline\supp_X(([f]^{<\omega}(b),W(|f\!\restriction\! b|)(\tau)))=[f]^{<\omega}(b).
\end{equation*}
We then have $f(s)\leq_X f(t')$ for some $t'\in b$. In view of $l(t')<l(t)$ the induction hypothesis yields $s\leq_{\T W}t'$. By clause~(i) of Definition~\ref{def:T(W)} we get $s\leq_{\T W}t$.
\end{proof}

Using the theorem, we can finally describe $\T W$ in a concrete case:

\begin{example}\label{ex:higman}
In Examples~\ref{ex:po-dilator-aca} and~\ref{ex:W_Z-normal} we have considered normal PO-dilators~$W_Z$ with $W_Z(X)=1+Z\times X$. We want to show that $\T W_Z$ is isomorphic to the set $\seq(Z)$ of finite sequences with entries in the partial order~$Z$, ordered as in Higman's lemma~\cite{higman52}. In this order we have $\langle z_0\dots,z_{m-1}\rangle\leq_{\seq(Z)}\langle z_0'\dots,z_{n-1}'\rangle$ if, and only if, there is a strictly increasing function $f:m=\{0,\dots,m-1\}\to\{0,\dots,n-1\}=n$ such that $z_i\leq_Z z'_{f(i)}$ holds for all $i<m$. A bijection $\kappa:W_Z(\seq(Z))\to\seq(Z)$ can be given by
\begin{equation*}
\kappa(0)=\langle\rangle\quad\text{and}\quad\kappa((z,\langle z_0,\dots,z_{n-1}\rangle))=\langle z,z_0,\dots,z_{n-1}\rangle,
\end{equation*}
where $0$ is the unique element of $1\subseteq W_Z(\seq(Z))$ and $\langle\rangle$ is the empty sequence. It is straightforward to verify the equivalence from Definition~\ref{def:Kruskal-fixed-point}. Strictly speaking, this equivalence should hold with respect to a function
\begin{equation*}
\overline\kappa:\overline{W_Z\!\restriction\!\poz}(\seq(Z))\to\seq(Z),
\end{equation*}
where $W_Z\!\restriction\!\poz$ is the coded restriction of $W_Z$ and $\overline{W_Z\!\restriction\!\poz}$ is its class-sized reconstruction. According to Theorem~\ref{thm:reconstruct-class-dilator} and Example~\ref{ex:coded-class-iso}, there is an isomorphism
\begin{equation*}
\eta_{\seq(Z)}:\overline{W_Z\!\restriction\!\poz}(\seq(Z))\xrightarrow{\cong}W_Z(\seq(Z))
\end{equation*}
that preserves supports. We can conclude that $\seq(Z)$ and $\bar\kappa:=\kappa\circ\eta_{\seq(Z)}$ form a Kruskal fixed point of~$W_Z$. To show that $\seq(Z)$ is an initial fixed point, we consider another Kruskal fixed point $\kappa':W_Z(X)\to X$. The condition $f\circ\kappa=\kappa'\circ W_Z(f)$ from Definition~\ref{def:Kruskal-fixed-point} is equivalent to
\begin{align*}
f(\langle\rangle)&=f\circ\kappa(0)=\kappa'\circ W_Z(f)(0)=\kappa'(0),\\
f(\langle z_0,\dots,z_n\rangle)&=\begin{multlined}[t]f\circ\kappa((z_0,\langle z_1,\dots,z_n\rangle))=\kappa'\circ W_Z(f)((z_0,\langle z_1,\dots,z_n\rangle))=\\ =\kappa'((z_0,f(\langle z_1,\dots,z_n\rangle))).\end{multlined}
\end{align*}
Clearly, there is a unique function $f:\seq(Z)\to X$ that satisfies these recursive equations. A straightforward induction over sequences shows that $f(s)\leq_X f(t)$ implies $s\leq_{\seq(Z)} t$. Hence $\seq(Z)$ is an initial Kruskal fixed point of~$W_Z$, provably in~$\rca_0$. Theorem~\ref{thm:initial-fixed-point} and the uniqueness of initial objects yield $\T W_Z\cong\seq(Z)$.
\end{example}

The theory $\mathbf{\Pi^1_1\text{-}CA}_0$ extends $\rca_0$ by the principle of $\Pi^1_1$-comprehension. Working in this theory, we now prove the uniform Kruskal theorem that has been mentioned in the introduction. Note that Hasegawa~\cite{hasegawa97} has established the same result in a somewhat different setting. Theorem~\ref{thm:main} below, which is the main~result of our paper, shows that the use of $\Pi^1_1$-comprehension is necessary.

\begin{theorem}[$\pica_0$]\label{thm:Pi11-to-uniform-Kruskal}
If $W$ is a normal WPO-dilator, then $\T W$ is a wpo.
\end{theorem}
\begin{proof}
Aiming at a contradiction, we assume there is a bad sequence $f:\mathbb N\to\T W$, which means that we have $f(i)\not\leq_{\T W} f(j)$ for all $i<j$. Let us recall the length function $l:\T W\to\mathbb N$ from above. The famous proof method by Nash-Williams~\cite{nash-williams63} suggests to consider a bad sequence~$g:\mathbb N\to\T W$ with the following minimality property: If $h:\mathbb N\to\T W$ is another bad sequence and $i$ is the smallest number with~$g(i)\neq h(i)$, then we have $l(g(i))\leq l(h(i))$. To see that such a sequence exists, we consider the tree
\begin{equation*}
T=\{\langle h(0),\dots,h(n-1)\rangle\,|\,h:\mathbb N\to\T W\text{ a bad sequence and }n\in\mathbb N\}
\end{equation*}
of all finite sequences that can be extended into a bad sequence. Note that $T$ can be formed by $\Sigma^1_1$-comprehension, which is equivalent to $\Pi^1_1$-comprehension. The existence of~$f$ ensures that~$T$ is non-empty. It is clear by construction that~$T$ has no leaves. We say that a sequence $\langle\sigma_0,\dots,\sigma_{n-1}\rangle\in T$ is $l$-minimal if the G\"odel number of each entry $\sigma_i$ is minimal with the following property: For any $\sigma_i'\in\T W$ with $l(\sigma_i')<l(\sigma_i)$ we have $\langle\sigma_0,\dots,\sigma_{i-1},\sigma_i'\rangle\notin T$. A straightforward induction on~$n$ shows that~$T$ contains a unique $l$-minimal sequence of each length~$n$. To construct the minimal bad sequence $g:\mathbb N\to\T W$ that was promised above, we declare that $g(n)$ is the last entry of the $l$-minimal sequence of length~$n+1$. Now write
\begin{equation*}
g(n)=\circ(a_n,\sigma_n)\in\T W\quad\text{and}\quad X=\bigcup\{a_n\,|\,n\in\mathbb N\}\subseteq\T W.
\end{equation*}
Let us show that~$(X,\leq_{\T W})$ is a well partial order. Aiming at a contradiction, we assume that $s_0,s_1,\dots$ is a bad sequence in~$X$. By recursion, we construct sequences of indices $i(0)<i(1)<\dots$ and $j(0)<j(1)<\dots$ with $s_{i(k)}\in a_{j(k)}$. For $k=0$ we put $i(k)=0$ and pick some index $j(0)$ with $s_0\in a_{j(0)}$. In the step we observe that
\begin{equation*}
\bigcup\{a_j\,|\,j\leq j(k)\}\subseteq X
\end{equation*}
is a finite set. Since $s_0,s_1,\dots$ is bad, there must be some index $i(k+1)>i(k)$ such that $s_{i(k+1)}$ does not lie in this set. In view of $s_{i(k+1)}\in X$ we get $s_{i(k+1)}\in a_{j(k+1)}$ for some $j(k+1)>j(k)$. Now consider the sequence
\begin{equation*}
g(0),g(1),\dots,g(j(0)-1),s_{i(0)},s_{i(1)},s_{i(2)},\ldots\subseteq\T W.
\end{equation*}
In view of $s_{i(0)}\in a_{j(0)}$ we have $l(s_{i(0)})<l(g(j(0)))$. Hence the minimality of~$g$ implies that the given sequence is good. For $j<j(0)$ and $k\in\mathbb N$ we observe that $g(j)\leq_{\T W} s_{i(k)}$ would imply $g(j)\leq_{\T W} g(j(k))$, due to $s_{i(k)}\in a_{j(k)}$. Since~$g$ is bad, this is impossible. We conclude that $s_0,s_1,\dots$ must be good, which contradicts our assumption. We have thus established that $X$ is a wpo. Then $\overline W(X)$ is a wpo as well, since $W$ was assumed to be a WPO-dilator. For $\kappa:\overline W(\T W)\to\T W$ as in the proof of Theorem~\ref{thm:initial-fixed-point}, we can write
\begin{equation*}
g(n)=\kappa((a_n,\sigma_n))\quad\text{with}\quad (a_n,\sigma_n)\in\overline W(\T W).
\end{equation*}
For each number $n$ we have $\overline\supp_{\T W}((a_n,\sigma_n))=a_n\subseteq X=\rng(\iota_X^{\T W})$, where we write $\iota_X^{\T W}:X\hookrightarrow\T W$ for the inclusion. In Theorem~\ref{thm:extend-po-dilator} we have shown that $\overline W$ is a class-sized PO-dilator, so that it satisfies the support condition from Definition~\ref{def:po-dilator}. The latter allows us to write
\begin{equation*}
(a_n,\sigma_n)=\overline W(\iota_X^{\T W})(\tau_n)\quad\text{with}\quad \tau_n\in\overline W(X).
\end{equation*}
Since $\overline W(X)$ is a well partial order, we find indices $i<j$ with $\tau_i\leq_{\overline W(X)}\tau_j$. Also due to Theorem~\ref{thm:extend-po-dilator} and Definition~\ref{def:po-dilator}, we know that $\overline W(\iota_X^{\T W})$ is an embedding. We thus get $(a_i,\sigma_i)\leq_{\overline W(\T W)}(a_j,\sigma_j)$ and then
\begin{equation*}
g(i)=\kappa((a_i,\sigma_i))\leq_{\T W} \kappa((a_j,\sigma_j))=g(j).
\end{equation*}
This contradicts the assumption that~$g$ is bad.
\end{proof}

\section{Uniform lower bounds for Kruskal-type theorems}\label{sect:lower-bounds}

In the previous section we have constructed a partial order~$\T W$ relative to a normal PO-dilator~$W$. The present section establishes a lower bound on the maximal order type of~$\T W$. This bound will have the form~$\vartheta(D)$ for a suitable LO-dilator~$D$, where $\vartheta(D)$ is the linear order defined by Freund~\cite{freund-computable}. Our proof of the lower bound also justifies the uniform independence principle that was stated in the introduction.

The notion of LO-dilator has been explained in the text after Definition~\ref{def:po-dilator}. We write $\loz$ for the category with objects $n=\{0,\dots,n-1\}$ and all strictly increasing functions $m\to n$ as morphisms. Note that $\loz$ is a subcategory of the category~$\poz$ that was considered in Section~\ref{sect:formalize-po-dilators}. As in the case of~PO-dilators, each class-sized LO-dilator~$D$ restricts to a coded LO-dilator~$D\!\restriction\!\loz$. Conversely, any coded LO-dilator~$D$ can be extended into a class-sized LO-dilator~$\overline D$, in such a way that we get $\overline{D\!\restriction\!\loz}\cong D$ for class-sized~$D$. For the linear case, these facts are due to Girard~\cite{girard-pi2}. A detailed presentation in our terminology is given in~\cite[Section~2]{freund-computable}. In the latter paper, LO-dilators are predominantly denoted by~$T$; the class-sized extension of a coded LO-dilator~$T$ is written as~$D^T$ rather than~$\overline T$.

We will see that $\T W$ can be bounded in terms of~$\vartheta(D)$ if $W$ and $D$ are related as in the following definition. Note that each coded PO-dilator~$W:\poz\to\po$ restricts to a functor~$W\!\restriction\!\loz:\loz\to\po$. We can also consider an LO-dilator $D:\loz\to\lo$ as a functor from~$\loz$ to~$\po$, leaving the inclusion~$\lo\hookrightarrow\po$ implicit. With respect to the following definition, this means that each component $\nu_n:D(n)\to W(n)$ with $n\in\loz$ is a morphism in~$\po$, i.\,e.~a quasi embedding. 

\begin{definition}[$\rca_0$]\label{def:quasi-embedding-dilators}
A quasi embedding of a coded LO-dilator~$D$ into a coded PO-dilator~$W$ is defined to be a natural transformation $\nu:D\Rightarrow W\!\restriction\!\loz$.
\end{definition}

If $D$ and $W$ are represented in~$\rca_0$, then the underlying sets of the orders~$D(n)$ and~$W(n)$ are contained in~$\mathbb N$. In this case, a quasi embedding $\nu:D\Rightarrow W\!\restriction\!\loz$ can be represented by the set
\begin{equation*}
\nu=\{\langle n,\sigma,\tau\rangle\,|\,n\in\mathbb N\text{ and }\sigma\in D(n)\text{ and }\tau=\nu_n(\sigma)\in W(n)\}.
\end{equation*}
We want to show that a quasi embedding between coded dilators induces a quasi embedding between their class-sized extensions. The following lemma is needed as a preparation. We write $\supp^D:D\Rightarrow[\cdot]^{<\omega}$ and $\supp^W:W\Rightarrow[\cdot]^{<\omega}$ for the supports that come with~$D$ and~$W$.

\begin{lemma}[$\rca_0$]\label{lem:transfos-preserve-supps}
If $\nu:D\Rightarrow W\!\restriction\!\loz$ is a quasi embedding, then the equation $\supp^W_n\circ\nu_n=\supp^D_n$ holds for each $n\in\loz$.
\end{lemma}
\begin{proof}
For natural transformations between LO-dilators this has been shown by Girard~\cite{girard-pi2}. A proof that uses our terminology can be found in~\cite[Lemma~2.17]{freund-rathjen_derivatives}. It is straightforward to check that this proof also applies when~$W$ is a PO-dilator.
\end{proof}

As in the case of PO-dilators, we write
\begin{equation*}
\trace(D)=\{(n,\sigma)\,|\,n\in\loz\text{ and }\sigma\in D(n)\text{ with }\supp^D_n(\sigma)=n\}
\end{equation*}
for the trace of an LO-dilator~$D$. According to~\cite[Definition~2.3]{freund-computable}, the class-sized extension of~$D$ is given by
\begin{equation*}
\overline D(X)=\{(a,\sigma)\,|\,a\in[X]^{<\omega}\text{ and }(|a|,\sigma)\in\trace(D)\},
\end{equation*}
for each linear order~$X$. The previous lemma ensures that $(n,\sigma)\in\trace(D)$ implies $(n,\nu_n(\sigma))\in\trace(W)$, which justifies the following construction.

\begin{definition}[$\rca_0$]
Consider a coded LO-dilator~$D$, a coded PO-dilator~$W$, and a quasi embedding $\nu:D\Rightarrow W$. For each linear order~$X$, we define a function $\overline\nu_X:\overline D(X)\Rightarrow \overline W(X)$ by setting $\overline\nu_X((a,\sigma))=(a,\nu_{|a|}(\sigma))$.
\end{definition}

Let us verify the expected property:

\begin{lemma}[$\rca_0$]\label{lem:quasi-embedding-extend}
Assume that $\nu:D\Rightarrow W$ is a quasi embedding. Then the functions $\overline\nu_X:\overline D(X)\Rightarrow\overline W(X)$ form a natural family of quasi embeddings. Furthermore we have $\overline\supp^W_X\circ\overline\nu_X=\overline\supp^D_X$ for each linear order~$X$.
\end{lemma}
\begin{proof}
We begin by showing that $\overline\nu_X$ is a quasi embedding. In view of Definition~\ref{def:extend-po-dilator}, an inequality $\overline\nu_X((a,\sigma))\leq_{\overline W(X)}\overline\nu_X((b,\tau))$ amounts to
\begin{equation*}
W(|\iota_a^{a\cup b}|)\circ\nu_{|a|}(\sigma)\leq_{W(|a\cup b|)} W(|\iota_b^{a\cup b}|)\circ\nu_{|b|}(\tau).
\end{equation*}
Due to the naturality of~$\nu$, the latter is equivalent to
\begin{equation*}
\nu_{|a\cup b|}\circ D(|\iota_a^{a\cup b}|)(\sigma)\leq_{W(|a\cup b|)} \nu_{|a\cup b|}\circ D(|\iota_b^{a\cup b}|)(\tau).
\end{equation*}
Since $\nu_{|a\cup b|}$ is a quasi embedding, we obtain
\begin{equation*}
D(|\iota_a^{a\cup b}|)(\sigma)\leq_{D(|a\cup b|)} D(|\iota_b^{a\cup b}|)(\tau).
\end{equation*}
First assume that we have equality. As in the proof of~\cite[Lemma~2.2]{freund-computable} (see~also the proof of Theorem~\ref{thm:extend-po-dilator} above), we can deduce $a=b$ and $\sigma=\tau$. Then reflexivity in $\overline D(X)$ yields $(a,\sigma)\leq_{\overline D(X)} (b,\tau)$. If we have $D(|\iota_a^{a\cup b}|)(\sigma)<_{D(|a\cup b|)} D(|\iota_b^{a\cup b}|)(\tau)$, then~$(a,\sigma)<_{\overline D(X)}(b,\tau)$ holds according to~\cite[Definition~2.2]{freund-computable}. Let us now establish naturality. Given a quasi embedding $f:X\to Y$ of linear orders, we observe
\begin{equation*}
\overline W(f)\circ\overline\nu_X((a,\sigma))=\overline W(f)((a,\nu_{|a|}(\sigma)))=([f]^{<\omega}(a),W(|f\!\restriction\!a|)\circ\nu_{|a|}(\sigma)).
\end{equation*}
Since~$X$ and~$Y$ are linear, we know that $f$ is in fact an embedding. Hence $a\subseteq X$ is isomorphic to $[f]^{<\omega}(a)\subseteq Y$. As each finite partial order has a unique representative in~$\poz$, it follows that the orders $|a|$ and $|[f]^{<\omega}(a)|$ are equal. The function $\en_{[f]^{<\omega}(a)}:|[f]^{<\omega}(a)|=|a|\to[f]^{<\omega}(a)$ is uniquely determined as the increasing enumeration of the order~$[f]^{<\omega}(a)\subseteq Y$, because the latter is linear. This allows us to conclude~$\en_{[f]^{<\omega}(a)}=(f\!\restriction\! a)\circ\en_a$. Also recall that $|f\!\restriction\!a|:|a|\to|[f]^{<\omega}(a)|$ is characterized as the unique function with $\en_{[f]^{<\omega}(a)}\circ |f\!\restriction\!a|=(f\!\restriction\! a)\circ\en_a$. It follows that $|f\!\restriction\!a|$ is the identity on~$|a|=|[f]^{<\omega}(a)|$. We now obtain
\begin{multline*}
\overline W(f)\circ\overline\nu_X((a,\sigma))=([f]^{<\omega}(a),W(|f\!\restriction\!a|)\circ\nu_{|a|}(\sigma))=([f]^{<\omega}(a),\nu_{|[f]^{<\omega}(a)|}(\sigma))=\\
=\overline\nu_Y(([f]^{<\omega}(a),\sigma))=\overline\nu_Y\circ\overline D(f)((a,\sigma)),
\end{multline*}
where the last equality relies on~\cite[Definition~2.2]{freund-computable}. Finally, we observe
\begin{equation*}
\overline\supp^W_X\circ\overline\nu_X((a,\sigma))=\overline\supp^W_X((a,\nu_{|a|}(\sigma)))=a=\overline\supp^D_X((a,\sigma)),
\end{equation*}
as we have claimed in the lemma.
\end{proof}

Recall that Freund~\cite{freund-computable} has defined a linear order~$\vartheta(D)$ for any LO-dilator~$D$. Elements of $\vartheta(D)$ have the form $\vartheta^a_\sigma$ with $a\in[\vartheta(D)]^{<\omega}$ and $(|a|,\sigma)\in\trace(D)$ (in~\cite{freund-computable} they are written as $\vartheta^{s_0,\dots,s_{n-1}}_\sigma$ for $a=\{s_0,\dots,s_{n-1}\}$ with $s_0<_{\vartheta(D)}\dots<_{\vartheta(D)}s_{n-1}$). The order relation is given by
\begin{equation*}
\vartheta^a_\sigma<_{\vartheta(D)}\vartheta^b_\tau\Leftrightarrow\begin{cases}
D(|\iota_a^{a\cup b}|)(\sigma)<_{D(|a\cup b|)} D(|\iota_b^{a\cup b}|)(\tau)\text{ and $s<_{\vartheta(D)}\vartheta^b_\tau$ for all $s\in a$},\\
\text{or $\vartheta^a_\sigma\leq_{\vartheta(D)}t$ for some $t\in b$}.
\end{cases}
\end{equation*}
As in the case of $\T W$, we define a length function $l:\vartheta(D)\to\mathbb N$ by
\begin{equation*}
l(\vartheta^a_\sigma)=\max\{\ulcorner\vartheta^a_\sigma\urcorner,1+\textstyle\sum_{s\in a}2\cdot l(s)\}.
\end{equation*}
We now come to one of the central observations of the present paper. The following result implies that $\vartheta(D)$ is a lower bound for~$\T W$, as we shall see below.

\begin{theorem}[$\rca_0$]\label{thm:quasi-embedding}
Consider an LO-dilator~$D$ and a normal PO-dilator~$W$. If there is a quasi embedding of~$D$ into~$W$, then there is a quasi embedding of the linear order~$\vartheta(D)$ into the partial order~$\T W$. 
\end{theorem}
\begin{proof}
Assume that $\nu:D\Rightarrow W$ is a quasi embedding. We define $f:\vartheta(D)\to\T W$ by recursion over the terms in~$\vartheta(D)$, setting
\begin{equation*}
f(\vartheta^a_\sigma):=\circ([f]^{<\omega}(a),W(|f\!\restriction\!a|)\circ\nu_{|a|}(\sigma)).
\end{equation*}
To show that the recursion is successful, we verify
\begin{align*}
r\in\vartheta(D)\,&\Rightarrow\,f(r)\in\T W,\\
f(s)\leq_{\mathcal T W} f(t)\,&\Rightarrow\, s\leq_{\vartheta(D)} t
\end{align*}
by simultaneous induction on $l(r)$ and $l(s)+l(t)$, respectively. To establish the first implication we write $r=\vartheta^a_\sigma$. Given $(|a|,\sigma)\in\trace(D)$, we can invoke Lemma~\ref{lem:transfos-preserve-supps} to get $(|a|,\nu_{|a|}(\sigma))\in\trace(W)$. The simultaneous induction hypothesis ensures that the function $f\!\restriction\!a:a\to[f]^{<\omega}(a)\subseteq\T W$ is a quasi embedding. Then Lemma~\ref{lem:trace-preserved} yields
\begin{equation*}
(|[f]^{<\omega}(a)|,W(|f\!\restriction\!a|)\circ\nu_{|a|}(\sigma))\in\trace(W).
\end{equation*}
In view of Definition~\ref{def:T(W)} we get $f(r)=f(\vartheta^a_\sigma)\in\T W$, as desired. Let us now consider an inequality $f(s)\leq_{\T W} f(t)$ with $s=\vartheta^a_\sigma$ and $t=\vartheta^b_\tau$. First assume that the latter holds by clause~(i) of Definition~\ref{def:T(W)}, which means that we have
\begin{equation*}
f(s)\leq_{\T W}t'\quad\text{for some $t'\in[f]^{<\omega}(b)$}.
\end{equation*}
For $t'=f(t'')$ with $t''\in b$, the induction hypothesis yields $s\leq_{\vartheta(D)}t''$. This implies $s<_{\vartheta(D)}t$, by the definition of the order on~$\vartheta(D)$. Now assume that $f(s)\leq_{\T W} f(t)$ holds by clause~(ii) of Definition~\ref{def:T(W)}, which amounts to
\begin{equation*}
W\left(\left|\iota_{[f]^{<\omega}(a)}^{[f]^{<\omega}(a\cup b)}\right|\circ\left|f\!\restriction\! a\right|\right)\circ\nu_{|a|}(\sigma)\leq_{W(|[f]^{<\omega}(a\cup b)|)}
W\left(\left|\iota_{[f]^{<\omega}(b)}^{[f]^{<\omega}(a\cup b)}\right|\circ\left|f\!\restriction\! b\right|\right)\circ\nu_{|b|}(\tau).
\end{equation*}
Due to the induction hypothesis, we know that $f\!\restriction\!(a\cup b):a\cup b\to [f]^{<\omega}(a\cup b)$ is a quasi embedding. As in the proof of Theorem~\ref{thm:extend-po-dilator}, we can then infer
\begin{equation*}
W(|\iota_a^{a\cup b}|)\circ\nu_{|a|}(\sigma)\leq_{W(|a\cup b|)} W(|\iota_b^{a\cup b}|)\circ\nu_{|b|}(\tau).
\end{equation*}
By naturality, this amounts to $\nu_{|a\cup b|}\circ D(|\iota_a^{a\cup b}|)(\sigma)\leq_{W(|a\cup b|)}\nu_{|a\cup b|}\circ D(|\iota_b^{a\cup b}|)(\tau)$. Since $\nu_{|a\cup b|}:D(|a\cup b|)\to W(|a\cup b|)$ is a quasi embedding, we obtain
\begin{equation*}
D(|\iota_a^{a\cup b}|)(\sigma)\leq_{D(|a\cup b|)}D(|\iota_b^{a\cup b}|)(\tau).
\end{equation*}
If we have equality, then we get $a=b$ and $\sigma=\tau$, as in the proof of the previous lemma. In this case, $s=\vartheta^a_\sigma\leq_{\vartheta(D)}\vartheta^b_\tau=t$ holds by reflexivity. Now assume that the above inequality is strict. In order to conclude $s<_{\vartheta(D)}t$, we need to establish $s'<_{\vartheta(D)}t$ for arbitrary $s'\in a$. Let us observe that $f(s')\in[f]^{<\omega}(a)$ implies
\begin{equation*}
f(s')\leq_{\T W}\circ([f]^{<\omega}(a),W(|f\!\restriction\!a|)\circ\nu_{|a|}(\sigma))=f(s)\leq_{\T W} f(t).
\end{equation*}
In view of $l(s')<l(s)$, the induction hypothesis yields $s'\leq_{\vartheta(D)} t$. To exclude equality, we deduce a contradiction from the assumption $s'=t$. The latter implies that we have $f(s)\leq_{\T W}f(t)=f(s')$. By Lemma~\ref{lem:order-heights} we get $h(f(s))\leq h(f(s'))$. However, in view of $f(s')\in[f]^{<\omega}(a)$ we also have
\begin{equation*}
h(f(s'))<h(\circ([f]^{<\omega}(a),W(|f\!\restriction\!a|)\circ\nu_{|a|}(\sigma)))=h(f(s)),
\end{equation*}
which yields the required contradiction.
\end{proof}

Following de Jongh and Parikh~\cite{deJongh-Parikh}, we write $o(X)$ for the maximal order type of a well partial order~$X$. The latter can be given as
\begin{equation*}
o(X)=\sup\{\alpha\,|\,\text{there is a quasi embedding of~$\alpha$ into~$X$}\},
\end{equation*}
where the ordinal~$\alpha$ is identified with its ordered set of predecessors. If $\T W$ is a well partial order, then the conclusion of the previous theorem implies that $\vartheta(D)$ is a well order with order type~$o(\vartheta(D))\leq o(\T W)$. A sound theory~$\mathbf T\supseteq\rca_0$ with proof theoretic ordinal at most $o(\vartheta(D))$ cannot prove that $\vartheta(D)$ is well founded (provided that $D$ and hence $\vartheta(D)$ is computable). This establishes the uniform independence principle that was stated towards the end of the introduction. A~similar argument yields the following result, which is useful because it allows us to work in the stronger base theory~$\aca_0$ of arithmetical comprehension. The result will eventually be superseded by Theorem~\ref{thm:main}.

\begin{lemma}[$\rca_0+\cac$]\label{lem:uniform-Kruskal_ACA}
Assume that $\T W$ is a well partial order whenever~$W$ is a normal WPO-dilator. Then arithmetical comprehension holds.
\end{lemma}
\begin{proof}
In Example~\ref{ex:po-dilator-aca} we have constructed a PO-dilator~$W_Z$ with
\begin{equation*}
W_Z(X)=1+Z\times X,
\end{equation*}
for each partial order~$Z$. From Example~\ref{ex:W_Z-normal} we know that $W_Z$ is normal. In the theory $\rca_0+\cac$ one can show that $W_Z$ (or rather its coded restriction) is a WPO-dilator whenever $Z$ is a well partial order, as discussed in Example~\ref{ex:coded-class-iso}. According to Example~\ref{ex:higman}, the Kruskal fixed point $\T W_Z$ is isomorphic to the order~$\seq(Z)$ of finite sequences with entries in~$Z$. Hence the assumption of the present result implies Higman's lemma. The latter is equivalent to arithmetical comprehension, as shown by Simpson~\cite{simpson-higman} and Girard~\cite{girard87}.
\end{proof}

\section{From the uniform Kruskal theorem to $\Pi^1_1$-comprehension}\label{sect:extend-dilator}

In this section we deduce $\Pi^1_1$-comprehension from the assumption that $\T W$ is a wpo whenever~$W$ is a WPO-dilator. Due to a result of Freund~\cite{freund-computable}, it suffices to establish that $\vartheta(D)$ is well founded for any given WO-dilator~$D$. For this purpose we construct a normal PO-dilator~$W_D$ and a quasi embedding~$\nu:D\Rightarrow W_D$. Our main technical result shows that $W_D$ preserves wpos. By the uniform Kruskal theorem we can conclude that $\T W_D$ is a well partial order. Then the quasi embedding $\vartheta(D)\to\T W_D$ from Theorem~\ref{thm:quasi-embedding} witnesses that $\vartheta(D)$ is well founded, as required.

To construct the aforementioned quasi embedding $\nu:D\Rightarrow W_D$, we will need to assume that $D$ satisfies a monotonicity property, which is due to Girard~\cite{girard-pi2}:

\begin{definition}[$\rca_0$]\label{def:monotone-dilators}
A coded LO-dilator~$D$ is called monotone if we have
\begin{equation*}
\text{$f(i)\leq g(i)$ for all~$i<m$}\quad\Rightarrow\quad\text{$D(f)(\sigma)\leq_{D(n)} D(g)(\sigma)$ for all~$\sigma\in D(m)$},
\end{equation*}
for all strictly increasing functions $f,g:m=\{0,\dots,m-1\}\to\{0,\dots,n-1\}=n$.
\end{definition}

Let us verify that the given property extends to infinite orders:

\begin{lemma}[$\rca_0$]
If $D$ is a monotone LO-dilator, then the following holds for all embeddings $f,g:X\to Y$ of linear orders: If $f(x)\leq_Y g(x)$ holds for all $x\in X$, then $\overline D(f)(\sigma)\leq_{\overline D(Y)}\overline D(g)(\sigma)$ holds for all~$\sigma\in\overline D(X)$.
\end{lemma}
\begin{proof}
A given element~$\sigma\in\overline D(X)$ can be written as $\sigma=(a,\sigma_0)$ for $a\in[X]^{<\omega}$ and $\sigma_0\in D(|a|)$ with $\supp_{|a|}(\sigma_0)=|a|$. In the following we consider the inclusions
\begin{equation*}
\iota_f:=\iota_{[f]^{<\omega}(a)}^{[f]^{<\omega}(a)\cup [g]^{<\omega}(a)}\quad\text{and}\quad \iota_g:=\iota_{[g]^{<\omega}(a)}^{[f]^{<\omega}(a)\cup [g]^{<\omega}(a)}.
\end{equation*}
By~\cite[Definition~2.2]{freund-computable}, the desired inequality $\overline D(f)(\sigma)\leq_{\overline D(Y)}\overline D(g)(\sigma)$ amounts to
\begin{equation*}
D(|\iota_f|)(\sigma_0)\leq_{D(|[f]^{<\omega}(a)\cup [g]^{<\omega}(a)|)} D(|\iota_g|)(\sigma_0).
\end{equation*}
Due to the implication in Definition~\ref{def:monotone-dilators}, this inequality reduces to the claim that
\begin{equation*}
|\iota_f|(i)\leq |\iota_g|(i)\quad\text{holds for all $i<|[f]^{<\omega}(a)|=|a|=|[g]^{<\omega}(a)|$}.
\end{equation*}
Since the increasing enumeration of $[f]^{<\omega}(a)\subseteq Y$ is unique, we have
\begin{equation*}
\en_{[f]^{<\omega}(a)\cup [g]^{<\omega}(a)}\circ |\iota_f|=\iota_f\circ\en_{[f]^{<\omega}(a)}=\iota_f\circ(f\!\restriction\! a)\circ\en_a,
\end{equation*}
and an analogous equation holds for~$g$. The required inequality $|\iota_f|(i)\leq |\iota_g|(i)$ is thus equivalent to $f(\en_a(i))\leq_Y g(\en_a(i))$, which holds by assumption.
\end{proof}

As shown by Girard~\cite[Proposition~2.3.10]{girard-pi2}, monotonicity is automatic in the well founded case. We translate Girard's proof into our terminology, for the reader's convenience and to ensure that the proof can be formalized in reverse mathematics.

\begin{lemma}[$\aca_0$]\label{lem:WO-dilator-monotone}
Any WO-dilator is monotone.
\end{lemma}
\begin{proof}
Recall that the ordinal $\omega^\omega$ can be represented by the set of finite non-increasing sequences of natural numbers. To suggest the intended interpretation as Cantor normal forms, we write this set as
\begin{equation*}
\omega^\omega=\{\omega^{n_0}+\dots+\omega^{n_{l-1}}\,|\,n_{l-1}\leq\dots\leq n_0<\omega\}.
\end{equation*}
In the appropriate order we have $\omega^{m_0}+\dots+\omega^{m_{k-1}}\preceq\omega^{n_0}+\dots+\omega^{n_{l-1}}$ if there is an $i<\min\{k,l\}$ with $m_i<n_i$ and $m_j=n_j$ for all~$j<i$, or if we have $k\leq l$ and $m_i=n_i$ for all $i<k$. The fact that $(\omega^\omega,\preceq)$ is a well order can be proved in~$\aca_0$ but not in~$\rca_0$ (see e.\,g.~\cite{kreuzer-yokoyama}). If $D$ is a WO-dilator, then~$\overline D(\omega^\omega)$ is well founded. To deduce that $D$ is monotone we consider strictly increasing functions $f,g:m\to n$ with $f(i)\leq g(i)$ for~$i<m$. The point of $\omega^\omega$ is that it admits strictly increasing functions $h:n\to\omega^\omega$ and $h':\omega^\omega\to\omega^\omega$ with $h\circ g=h'\circ h\circ f$. Before we justify this claim, we show how it allows us to conclude: Since $\overline D(\omega^\omega)$ is well founded and $\overline D(h')$ is an embedding, we have \begin{equation*}
\tau\leq_{\overline D(\omega^\omega)} \overline D(h')(\tau)\quad\text{for all $\tau\in\overline D(\omega^\omega)$}.
\end{equation*}
Indeed, if $\tau$ was minimal with $\tau>_{\overline D(\omega^\omega)}\overline D(h')(\tau)=:\tau'$, then we would get
\begin{equation*}
\overline D(h')(\tau)=\tau'\leq_{\overline D(\omega^\omega)}\overline D(h')(\tau')<_{\overline D(\omega^\omega)} \overline D(h')(\tau),
\end{equation*}
which is impossible. To establish normality we now deduce $D(f)(\sigma)\leq_{D(n)}D(g)(\sigma)$ for a given element $\sigma\in D(m)$. The latter can be written as $\sigma=D(\iota_a^m\circ\en_a)(\sigma_0)$ with $a=\supp_m(\sigma)$ and $\sigma_0\in D(|a|)$, due to the support condition in~\cite[Definition~2.1]{freund-computable}. As in the proof of Theorem~\ref{thm:reconstruct-class-dilator}, the naturality of supports ensures $\supp_{|a|}(\sigma_0)=|a|$ and hence $(a,\sigma_0)\in\overline D(m)$. By the above we get
\begin{equation*}
\overline D(h\circ f)((a,\sigma_0))\leq_{\overline D(\omega^\omega)}\overline D(h')\circ\overline D(h\circ f)((a,\sigma_0))=\overline D(h\circ g)((a,\sigma_0)),
\end{equation*}
which implies $\overline D(f)((a,\sigma_0))\leq_{\overline D(n)}\overline D(g)((a,\sigma_0))$. The latter is equivalent to
\begin{equation*}
D(|\iota^f|)(\sigma_0)\leq_{D(n)}D(|\iota^g|)(\sigma_0),
\end{equation*}
where $\iota^f:[f]^{<\omega}(a)\hookrightarrow n$ and $\iota^g:[g]^{<\omega}(a)\hookrightarrow n$ are the inclusions (note that we have $\iota^f=\iota\circ\iota_f$ for $\iota:[f]^{<\omega}(a)\cup[g]^{<\omega}(a)\hookrightarrow n$ and $\iota_f$ as in the previous proof). Since~$\en_n$ is the identity on $|n|=n=\{0,\dots,n-1\}$, we have
\begin{equation*}
|\iota^f|=\en_n\circ |\iota^f|=\iota^f\circ\en_{[f]^{<\omega}(a)}=\iota^f\circ(f\!\restriction\!a)\circ\en_a=f\circ\iota_a^m\circ\en_a,
\end{equation*}
as well as $|\iota^g|=g\circ\iota_a^m\circ\en_a$. Hence the above implies
\begin{multline*}
D(f)(\sigma)=D(f)\circ D(\iota_a^m\circ\en_a)(\sigma_0)=D(|\iota^f|)(\sigma_0)\leq_{D(n)}\\ \leq_{D(n)} D(|\iota^g|)(\sigma_0)=D(g)\circ D(\iota_a^m\circ\en_a)(\sigma_0)=D(g)(\sigma),
\end{multline*}
just as required. It remains to construct embeddings $h:n\to\omega^\omega$ and $h':\omega^\omega\to\omega^\omega$ such that $h\circ g(i)=h'\circ h\circ f(i)$ holds for all~$i<m$. We only consider the non-trivial case of~$m>0$. Recall that addition on $\omega^\omega$ can be represented by
\begin{equation*}
(\omega^{m_0}+\dots+\omega^{m_{k-1}})+(\omega^{n_0}+\dots+\omega^{n_{l-1}})=\omega^{m_0}+\dots+\omega^{m_{i-1}}+\omega^{n_0}+\dots+\omega^{n_{l-1}},
\end{equation*}
where $i$ is minimal with $m_i<n_0$ (take $i=k$ in case $m_{k-1}\geq n_0$ or $k=0$ or $l=0$). Basic facts of ordinal arithmetic are readily verified. For $\omega^{f(0)}\preceq\alpha\in\omega^\omega$ we now define $e(\alpha)\in\mathbb N$ and $r(\alpha)\in\omega^\omega$ by stipulating
\begin{equation*}
e(\alpha)=\max\{i<m\,|\,\omega^{f(i)}\preceq\alpha\}\quad\text{and}\quad \alpha=\omega^{f\circ e(\alpha)}+r(\alpha).
\end{equation*}
The desired functions $h:n\to\omega^\omega$ and $h':\omega^\omega\to\omega^\omega$ can then be defined by
\begin{equation*}
h(k)=\omega^k\quad\text{and}\quad h'(\alpha)=\begin{cases}
\omega^{g\circ e(\alpha)}+r(\alpha) &\text{if $\omega^{f(0)}\preceq\alpha$},\\
\alpha & \text{otherwise}.
\end{cases}
\end{equation*}
For $i<m$ we have $h\circ f(i)=\omega^{f(i)}$, which yields $e(h\circ f(i))=i$ and $r(h\circ f(i))=0$. Hence we get
\begin{equation*}
h'\circ h\circ f(i)=\omega^{g\circ e(h\circ f(i))}+r(h\circ f(i))=\omega^{g(i)}=h\circ g(i),
\end{equation*}
as desired. To show that $h'$ is strictly increasing, we first observe that $\alpha\prec\omega^{f(0)}\preceq\beta$ implies $h'(\alpha)=\alpha\prec\beta\preceq h'(\beta)$, where the last inequality relies on~$f\circ e(\beta)\leq g\circ e(\beta)$. For $\omega^{f(0)}\preceq\alpha\prec\beta$ we clearly have $e(\alpha)\leq e(\beta)$. If we have $e(\alpha)=e(\beta)$, then we get $r(\alpha)\prec r(\beta)$ and thus $h'(\alpha)\prec h'(\beta)$. Now assume that we have $e(\alpha)<e(\beta)$. By the maximality of $e(\alpha)$ we then get
\begin{equation*}
r(\alpha)\preceq\alpha\prec\omega^{f\circ e(\beta)}\preceq\omega^{g\circ e(\beta)}.
\end{equation*}
Since $\omega^{g\circ e(\beta)}\succ\omega^{g\circ e(\alpha)}$ is additively principal, we obtain
\begin{equation*}
h(\alpha)=\omega^{g\circ e(\alpha)}+r(\alpha)\prec\omega^{g\circ e(\beta)}\preceq\omega^{g\circ e(\beta)}+r(\beta)=h'(\beta),
\end{equation*}
as needed to show that $h'$ is strictly increasing.
\end{proof}

Our next goal is to extend an LO-dilator into a WO-dilator. Let us begin with some terminology: Given a partial order~$X$, we write $\emb(X)$ for the set of finite quasi embeddings $u:n\to X$, where $n=\{0,\dots,n-1\}$ carries the usual linear order. In this context we write $[u]=n$ for the domain of~$u$. For $u,w\in\emb(X)$ we define $\hig(u,w)$ as the set of strictly increasing functions $h:[u]\to[w]$ such that $u(i)\leq_X w\circ h(i)$ holds for all~$i<[u]$ (note the connection with Higman's lemma). In $\rca_0$ one should represent~$\emb(X)$ by the set of pairs $(a,u_0)$, where $a\subseteq X$ is a finite suborder and $u_0:n\to a$ is a surjective quasi embedding. Since $(a,u_0)$ corresponds to an obvious~$u:n\to X$, we will not make this representation explicit. The following definition can be made for any LO-dilator~$D$. However, we will need to assume that~$D$ is monotone to construct a quasi embedding $\nu:D\Rightarrow W_D$.

\begin{definition}[$\rca_0$]\label{def:extend-lo-dilators}
Let $D$ be an LO-dilator. For each partial order~$X$ we define a set $W_D(X)$ and a relation $\leq_{W_D(X)}$ by stipulating
\begin{gather*}
W_D(X)=\{(u,\sigma)\,|\,u\in\emb(X)\text{ and }([u],\sigma)\in\trace(D)\},\\
(u,\sigma)\leq_{W_D(X)}(w,\tau)\,\Leftrightarrow\,\text{there is an $h\in\hig(u,w)$ with $D(h)(\sigma)\leq_{D([w])}\tau$}.
\end{gather*}
Given a quasi embedding $f:X\to Y$, we define $W_D(f):W_D(X)\to W_D(Y)$ by
\begin{equation*}
W_D(f)((u,\sigma))=(f\circ u,\sigma).
\end{equation*}
Finally, we define functions $\supp^W_X:W_D(X)\to[X]^{<\omega}$ by setting
\begin{equation*}
\supp^W_X((u,\sigma))=\rng(u)=[u]^{<\omega}([u])
\end{equation*}
for each element $(u,\sigma)\in W_D(X)$.
\end{definition}

Let us point out that our order $W_D(X)$ is similar to the order~$Q(X)$ of K\v{r}i\v{z} and Thomas~\cite{kriz-thomas}, if we take $Q$ to be the category with objects~$\trace(D)$ and a suitable set of morphisms. However, the first component of an element $(u,\sigma)\in Q(X)$ can be an arbitrary function $u:[u]\to X$, while we restrict to quasi embeddings.

If the construction from Definition~\ref{def:extend-lo-dilators} is restricted to the category~$\poz$, then it can be represented as a single set, which is available in~$\rca_0$ (cf.~Example~\ref{ex:coded-po-dilator}). The following result is concerned with this set-sized restriction.

\begin{proposition}[$\rca_0$]\label{prop:W_D-PO-dilator}
If $D$ is an LO-dilator, then~$W_D$ is a normal PO-dilator.
\end{proposition}
\begin{proof}
We first show that $W_D(a)$ is a partial order for any~$a\in\poz$. Concerning reflexivity, we observe that $(u,\sigma)\leq_{W_D(a)}(u,\sigma)$ is witnessed by the identity on~$[u]$. In order to establish antisymmetry, we assume that the inequalities
\begin{equation*}
(u,\sigma)\leq_{W_D(a)}(w,\tau)\quad\text{and}\quad(w,\tau)\leq_{W_D(a)}(u,\sigma)
\end{equation*}
are witnessed by $h\in\hig(u,w)$ and $h'\in\hig(w,u)$. As the functions $h:[u]\to[w]$ and $h':[w]\to[u]$ are strictly increasing, they must both be the identity on~$[u]=[w]$. In view of $u(i)\leq_a w\circ h(i)=w(i)$ and $w(i)\leq_a u\circ h'(i)=u(i)$ we get $u=w$. Since~$D$ is a functor, we also obtain $\sigma=D(h)(\sigma)\leq_{D([u])}\tau$ and $\tau=D(h')(\tau)\leq_{D([u])}\sigma$. Hence antisymmetry in $D([u])$ yields $\sigma=\tau$. For transitivity we assume that
\begin{equation*}
(u,\sigma)\leq_{W_D(a)}(w,\tau)\quad\text{and}\quad(w,\tau)\leq_{W_D(a)}(v,\rho)
\end{equation*}
are witnessed by $h\in\hig(u,w)$ and $h'\in\hig(w,v)$. Then we have~$h'\circ h\in\hig(u,v)$. As $D(h'):D([w])\to D([v])$ is an embedding, we also get
\begin{equation*}
D(h'\circ h)(\sigma)=D(h')\circ D(h)(\sigma)\leq_{D([v])}D(h')(\tau)\leq_{D([v])}\rho.
\end{equation*}
Hence $h'\circ h$ witnesses $(u,\sigma)\leq_{W_D(a)}(v,\rho)$. Let us now discuss the action of~$W_D$ on morphisms. Given a quasi embedding $f:a\to b$, we consider an inequality
\begin{equation*}
(f\circ u,\sigma)=W_D(f)((u,\sigma))\leq_{W_D(b)} W_D(f)((w,\tau))=(f\circ w,\tau).
\end{equation*}
Assume that the latter is witnessed by $h\in\hig(f\circ u,f\circ w)$. Recall that $[v]$ denotes the domain of~$v$. Hence we have $[f\circ u]=[u]$ and $[f\circ w]=[w]$. Since~$f$ is a quasi embedding, we see that $f\circ u(i)\leq_b f\circ w\circ h(i)$ implies $u(i)\leq_a w\circ h(i)$. Thus we get~$h\in\hig(u,w)$, and this function witnesses $(u,\sigma)\leq_{W_D(a)}(w,\tau)$. If~$f$ is an embedding, then $h\in\hig(u,w)$ does also imply $h\in\hig(f\circ u,f\circ w)$. One can conclude that $W_D(f)$ is an embedding, as required by condition~(i) of  Definition~\ref{def:coded-po-dilator}. It is straightforward to verify that~$W_D$ is functorial. To see that $\supp^W$ is a natural transformation we consider a quasi embedding $f:a\to b$ and compute
\begin{multline*}
\supp^W_b\circ W_D(f)((u,\sigma))=\supp^W_b((f\circ u,\sigma))=[f\circ u]^{<\omega}([f\circ u])=\\
=[f]^{<\omega}\circ[u]^{<\omega}([u])=[f]^{<\omega}\circ\supp^W_a((u,\sigma)).
\end{multline*}
Let us now establish the support condition: Given an embedding~$f:a\to b$ and an element $(w,\sigma)\in W_D(b)$ with
\begin{equation*}
\rng(w)=\supp^W_b((w,\sigma))\subseteq\rng(f),
\end{equation*}
we find a function $u:[u]=[w]\to a$ with $f\circ u=w$. In order to show that $u$ is a quasi embedding, we consider an inequality $u(i)\leq_a u(j)$ with $i,j<[u]$. Since~$f$ is an embedding, we obtain $w(i)=f\circ u(i)\leq_b f\circ u(j)=w(j)$ and then $i\leq j$. Due to $u\in\emb(a)$ we now get $(u,\sigma)\in W_D(a)$. By construction we have
\begin{equation*}
(w,u)=(f\circ u,\sigma)=W_D(f)((u,\sigma))\in\rng(W_D(f)),
\end{equation*}
as required. Finally, we establish the normality condition from Definition~\ref{def:PO-dilator-normal}: Consider an inequality $(u,\sigma)\leq_{W_D(a)}(w,\tau)$ that is witnessed by $h\in\hig(u,w)$. An arbitrary element of $\supp^W_a((u,\sigma))$ has the form $u(i)$ with $i<[u]$. We have
\begin{equation*}
u(i)\leq_a w\circ h(i)\in\rng(w)=\supp^W_a((w,\tau)).
\end{equation*}
Hence~$h$ ensures $\supp^W_a((u,\sigma))\leqf_a\supp^W_a((w,\tau))$, as required.
\end{proof}

As promised, monotonicity allows us to view $W_D$ as an extension of~$D$:

\begin{proposition}[$\rca_0$]\label{prop:D-into-W_D}
Assume that $D$ is a monotone LO-dilator. Then there is a quasi embedding $\nu:D\Rightarrow W_D$.
\end{proposition}
\begin{proof}
In view of Definition~\ref{def:quasi-embedding-dilators} we need to define a natural family of quasi embeddings $\nu_m:D(m)\to W_D(m)$ for the linear orders~$m=\{0,\dots,m-1\}$. Due to the support condition from~\cite[Definition~2.1]{freund-computable}, each element $\sigma\in D(m)$ can be written as $\sigma=D(\iota_a^m\circ\en_a)(\sigma_0)$ with $a=\supp^D_m(\sigma)$ and $\sigma_0\in D(|a|)$, where $\iota_a^m:a\hookrightarrow m$ is the inclusion and $\en_a:|a|=\{0,\dots,|a|-1\}\to a\subseteq m$ is the strictly increasing enumeration. Note that $\sigma_0$ is unique since $D(\iota_a^m\circ\en_a)$ is an embedding. Due to
\begin{equation*}
a=\supp^D_m(D(\iota_a^m\circ\en_a)(\sigma_0))=[\iota_a^m\circ\en_a]^{<\omega}(\supp^D_{|a|}(\sigma_0))
\end{equation*}
we get $\supp^D_{|a|}(\sigma_0)=|a|$ and hence $(|a|,\sigma_0)\in\trace(D)$. Also note that we can view $\iota_a^m\circ\en_a$ as an element of $\emb(m)$, with $[\iota_a^m\circ\en_a]=|a|$. Thus we may set
\begin{equation*}
\nu_m(\sigma)=(\iota_a^m\circ\en_a,\sigma_0)\in W_D(m)\quad\text{for $\sigma=D(\iota_a^m\circ\en_a)(\sigma_0)$ with $a=\supp^D_m(\sigma)$}.
\end{equation*}
To show that $\nu_m$ is a quasi embedding we consider an inequality
\begin{equation*}
(\iota_a^m\circ\en_a,\sigma_0)=\nu_m(\sigma)\leq_{W_D(m)}\nu_m(\tau)=(\iota_b^m\circ\en_b,\tau_0)
\end{equation*}
with $\sigma=D(\iota_a^m\circ\en_a)(\sigma_0)$ and $\tau=D(\iota_b^m\circ\en_b)(\tau_0)$. According to Definition~\ref{def:extend-lo-dilators} there is a function $h\in\hig(\iota_a^m\circ\en_a,\iota_b^m\circ\en_b)$ with $D(h)(\sigma_0)\leq_{D(|b|)}\tau_0$. Now we use the assumption that~$D$ is monotone. Since we have $\iota_a^m\circ\en_a(i)\leq\iota_b^m\circ\en_b\circ h(i)$ for all $i<|a|$, it allows us to conclude
\begin{equation*}
\sigma=D(\iota_a^m\circ\en_a)(\sigma_0)\leq_{D(m)}D(\iota_b^m\circ\en_b)\circ D(h)(\sigma_0)\leq_{D(m)}D(\iota_b^m\circ\en_b)(\tau_0)=\tau.
\end{equation*}
To show that $\nu$ is natural we consider a strictly increasing function~$f:m\to n$ and an element $\sigma\in D(m)$. As above, we write $\sigma=D(\iota_a^m\circ\en_a)(\sigma_0)$ with $a=\supp^D_m(\sigma)$. Since $(f\!\restriction\!a)\circ\en_a:|a|\to [f]^{<\omega}(a)=:b$ is strictly increasing and surjective, we get
\begin{equation*}
f\circ\iota_a^m\circ\en_a=\iota_b^n\circ(f\!\restriction\!a)\circ\en_a=\iota_b^n\circ\en_b.
\end{equation*}
Hence we have $D(f)(\sigma)=D(\iota_b^n\circ\en_b)(\sigma_0)$, as well as
\begin{equation*}
\supp^D_n(D(f)(\sigma))=[f]^{<\omega}(\supp^D_m(\sigma))=b.
\end{equation*}
Invoking the definition of~$\nu$, we can infer
\begin{equation*}
\nu_n\circ D(f)(\sigma)=(\iota_b^n\circ\en_b,\sigma_0)=(f\circ\iota_a^m\circ\en_a,\sigma_0)=W_D(f)\circ\nu_m(\sigma),
\end{equation*}
as required for naturality.
\end{proof}

In the following we prove a main technical result of our paper, which states that $W_D$ preserves well partial orders whenever $D$ preserves well orders. To begin, we note that $\emb(X)$ can be seen as a subset of the set of finite sequences with entries from the partial order~$X$. The order from Higman's lemma can be given as
\begin{equation*}
u\leq_H w\quad:\Leftrightarrow\quad\hig(u,w)\neq\emptyset
\end{equation*}
for $u,w\in\emb(X)$. Assuming $u\leq_H w$, we construct a minimal $h[u,w]\in\hig(u,w)$ by recursion: For $i<[u]$ we define $h[u,w](i)$ as the smallest $j>h[u,w](i-1)$ with $u(i)\leq_X w(j)$ (read $h[u,w](-1)=-1$ to cover the case $i=0$). A straightforward induction on~$i$ shows that we have
\begin{equation*}
h[u,w](i)\leq h(i)\quad\text{for any $h\in\hig(u,w)$ and any $i<[u]$}.
\end{equation*}
The following result suggests a strategy to prove that $W_D(X)$ is a well partial order.

\begin{proposition}[$\rca_0$]\label{prop:extension-wpo}
Consider a WO-dilator~$D$ and a partial order~$X$, as well as an infinite sequence
\begin{equation*}
(u_0,\sigma_0),(u_1,\sigma_1),\ldots\subseteq W_D(X)\quad\text{ with }\quad u_0\leq_H u_1\leq_H\ldots.
\end{equation*}
Assume that there is a wpo~$Z$ and a family of quasi embeddings $u^i:[u_i]\to Z$ with $u^j\circ h[u_i,u_j]=u^i$ for all~$i<j$. Then there are $i<j$ with $(u_i,\sigma_i)\leq_{W_D(X)}(u_j,\sigma_j)$.
\end{proposition}
\begin{proof}
For $i\leq j$ we abbreviate $h_{ij}:=h[u_i,u_j]$. Note that $h_{ii}$ is the identity on~$[u_i]$. The assumptions of the proposition imply $h_{jk}\circ h_{ij}=h_{ik}$ for $i\leq j\leq k$, via
\begin{equation*}
u^k\circ h_{jk}\circ h_{ij}=u^j\circ h_{ij}=u^i=u^k\circ h_{ik}.
\end{equation*}
In~$\rca_0$, the direct limit of the embeddings $h_{ij}:[u_i]\to[u_j]$ can be given as
\begin{gather*}
Y=\{(j,n)\,|\,n<[u_j]\text{ and $n\notin\rng(h_{ij})$ for all $i<j$}\},\\
(i,m)\leq_Y(j,n)\,\Leftrightarrow\,h_{ik}(m)\leq h_{jk}(n)\text{ with }k=\max\{i,j\}.
\end{gather*}
It is straightforward to check that~$Y$ is a linear order. Furthermore, a family of embeddings $w_i:[u_i]\to Y$ with $w_j\circ h_{ij}=w_i$ for $i\leq j$ can be defined by
\begin{equation*}
w_i(m)=(j,n)\quad\text{with $j=\min\{j\leq i\,|\,m\in\rng(h_{ji})\}$ and $m=h_{ji}(n)$}.
\end{equation*}
Note that any $(i,m)\in Y$ arises as $(i,m)=w_i(m)\in\rng(w_i)$. If we define $u:Y\to Z$ by $u((i,m))=u^i(m)$, then we get $u\circ w_i=u^i$ for all~$i\in\mathbb N$. One readily verifies that $u$ is a quasi embedding. Since $Z$ is a well partial order, it follows that~$Y$ is a well order. Recall that~$\overline D$ denotes the class-sized extension of the coded LO-dilator~$D$ (cf.~\cite[Section~2]{freund-computable}, where $\overline T$ is written as $D^T$). Given that $D$ is a WO-dilator, we learn that $\overline D(Y)$ is a well order. Put $a_i:=\rng(w_i)\in[Y]^{<\omega}$ and observe $|a_i|=[u_i]$. As $(u_i,\sigma_i)\in W_D(X)$ entails $([u_i],\sigma_i)\in\trace(D)$, we can consider the sequence
\begin{equation*}
(a_0,\sigma_0),(a_1,\sigma_0),\ldots\subseteq\overline D(Y).
\end{equation*}
Since $\overline D(Y)$ is a well order, we get indices $i<j$ with $(a_i,\sigma_i)\leq_{\overline D(Y)}(a_j,\sigma_j)$. In~view of~\cite[Definition~2.2]{freund-computable} and $a_i\subseteq a_j=a_i\cup a_j$, this inequality amounts to
\begin{equation*}
D(|\iota|)(\sigma_i)\leq_{D(|a_j|)}\sigma_j\quad\text{for the inclusion $\iota:a_i\hookrightarrow a_j$}.
\end{equation*}
Recall that the function $|\iota|:|a_i|\to|a_j|$ is uniquely characterized by $\en_j\circ|\iota|=\iota\circ\en_i$, where $\en_i:|a_i|\to a_i$ is the increasing enumeration. If we write $\iota_i:a_i\hookrightarrow Y$ for the inclusion, then the function $\iota_i\circ\en_i:|a_i|=[u_i]\to Y$ is strictly increasing with range $a_i$, so that it must coincide with $w_i$. Hence we get
\begin{equation*}
\iota_j\circ\en_j\circ h_{ij}=w_j\circ h_{ij}=w_i=\iota_i\circ\en_i=\iota_j\circ\iota\circ\en_i.
\end{equation*}
This yields $\en_j\circ h_{ij}=\iota\circ\en_i$ and then $h_{ij}=|\iota|$. We thus get $D(h_{ij})(\sigma_i)\leq_{D([u_j])}\sigma_j$, which means that $h_{ij}\in\hig(u_i,u_j)$ witnesses $(u_i,\sigma_i)\leq_{W_D(X)}(u_j,\sigma_j)$.
\end{proof}

The first part of the previous proof suggests to single out the following notion.

\begin{definition}[$\rca_0$]
For a partial order $X$, a sequence $u_0,u_1,\ldots\subseteq\emb(X)$ is called directed if we have $u_0\leq_H u_1\leq_H\ldots$ and $h[u_i,u_k]=h[u_j,u_k]\circ h[u_i,u_j]$ for all $i<j<k$.
\end{definition}

As the next result shows, a directed sequence is all we need in order to satisfy the assumptions of Proposition~\ref{prop:extension-wpo}. In the following we consider $\emb(X)$ with the partial order~$\leq_H$ from Higman's lemma. The latter is provable in $\aca_0$ (due to Simpson~\cite{simpson-higman}) and ensures that $\emb(X)$ is a wpo whenever the same holds for~$X$.

\begin{proposition}[$\rca_0$]\label{lem:directed-sequence-limit}
For any directed sequence $u_0,u_1,\ldots\subseteq\emb(X)$ there is a family of quasi embeddings $u^k:[u_k]\to\emb(X)$ with $u^l\circ h[u_k,u_l]=u^k$~for~$k<l$.
\end{proposition}
\begin{proof} 
As before, we abbreviate $h_{ij}:=h[u_i,u_j]$ for $i\leq j$. Since $h_{kk}$ is the identity, any number $n<[u_k]$ can be written as $n=h_{ik}(m)$ with $i\leq k$ and $m<[u_i]$. Note that $m$ is uniquely determined once a choice for~$i$ has been fixed. We define
\begin{equation*}
u^k(n)=u_i\!\restriction\!(m+1)\quad\text{for $n=h_{ik}(m)$ with $i\leq k$ as small as possible}.
\end{equation*}
To avoid confusion we emphasize that $u_i\in\emb(X)$ is a function with codomain~$X$ while $u^k$ is a function with codomain~$\emb(X)$. For $k<l$ we see that $n=h_{ik}(m)$ implies $h_{kl}(n)=h_{kl}\circ h_{ik}(m)=h_{il}(m)$. Furthermore, if we had $j<i$ and $m'<[u_j]$ with $h_{kl}(n)=h_{jl}(m')=h_{kl}\circ h_{jk}(m')$, then the fact that $h_{kl}$ is strictly increasing would yield $n=h_{jk}(m')$, against the minimality of~$i$ in $n=h_{ik}(m)$. Hence we get
\begin{equation*}
u^l\circ h_{kl}(n)=u^l\circ h_{il}(m)=u_i\!\restriction\!(m+1)=u^k\circ h_{ik}(m)=u^k(n).
\end{equation*}
It remains to show that each function $u^k:[u_k]\to\emb(X)$ is a quasi embedding. For this purpose we consider an inequality $u^k(n)\leq_H u^k(n')$. Writing $n=h_{ik}(m)$ and $n'=h_{jk}(m')$ with $i$ and $j$ as small as possible, the latter amounts to
\begin{equation*}
u_i\!\restriction\!(m+1)\leq_H u_j\!\restriction\!(m'+1).
\end{equation*}
This inequality is witnessed by a strictly increasing $h:\{0,\dots,m\}\to\{0,\dots,m'\}$ such that $u_i(l)\leq_X u_j(h(l))$ holds for all~$l\leq m$. Let us first assume $j\leq i$. In this case we observe that $m\leq h(m)\leq m'\leq h_{ji}(m')$ implies
\begin{equation*}
n=h_{ik}(m)\leq h_{ik}\circ h_{ji}(m')=h_{jk}(m')=n'.
\end{equation*}
Now assume~$i<j$. We recall that $h_{ij}=h[u_i,u_j]\in\hig(u_i,u_j)$ was defined as the element with the smallest possible values. By induction on $l\leq m<[u_i]$ one can deduce~$h_{ij}(l)\leq h(l)$. For $l=m$ this yields $h_{ij}(m)\leq h(m)\leq m'$ and then
\begin{equation*}
n=h_{ik}(m)=h_{jk}\circ h_{ij}(m)\leq h_{jk}(m')=n',
\end{equation*}
as needed to show that $u^k$ is a quasi embedding.
\end{proof}

Finally, we satisfy the precondition of the previous proposition:

\begin{proposition}[$\aca_0$]\label{prop:directed-subsequence}
If $X$ is a well partial order, then any infinite sequence in~$\emb(X)$ has a directed subsequence.
\end{proposition}
\begin{proof}
Recall that~$\aca_0$ proves Higman's lemma, as well as the infinite Ramsey theorem for each finite exponent (see~e.\,g.~\cite{simpson09}). We want to construct a directed subsequence of a given sequence $w_0,w_1,\ldots\subseteq\emb(X)$. Ramsey's theorem for pairs yields a strictly increasing function $f:\mathbb N\to\mathbb N$ such that one of the following~holds: Either we have $w_{f(i)}\leq_H w_{f(j)}$ for all $i<j$, or we have $w_{f(i)}\not\leq_H w_{f(j)}$ for all~$i<j$. The latter is excluded by Higman's lemma, given that~$X$ is a well partial order. With $u_i:=w_{f(i)}$ it suffices to construct a directed subsequence of $u_0\leq_H u_1\leq_H...$. For $i\leq j$ we abbreviate $h_{ij}:=h[u_i,u_j]\in\hig(u_i,u_j)$ and set
\begin{equation*}
o_i(j):=\langle u_j\!\restriction\!(h_{ij}(0)+1),\dots,u_j\!\restriction\!(h_{ij}([u_i]-1)+1)\rangle\in\emb(X)^{[u_i]}.
\end{equation*}
Write $\preceq_i$ for the usual partial order on the product $\emb(X)^{[u_i]}$, so that we have
\begin{equation*}
o_i(j)\preceq_i o_i(k)\quad\Leftrightarrow\quad u_j\!\restriction\!(h_{ij}(l)+1)\leq_H u_k\!\restriction\!(h_{ik}(l)+1)\text{ for all $l<[u_i]$}.
\end{equation*}
Given $i<j<k$, we now put
\begin{equation*}
c(i,j,k)=\begin{cases}
1 & \text{if $o_i(j)\preceq_i o_i(k)$},\\
0 & \text{otherwise}.
\end{cases}
\end{equation*}
By Ramsey's theorem (for exponent~$3$) there is a value $i_0\in\{0,1\}$ and a strictly increasing function $g:\mathbb N\to\mathbb N$ with $c(g(i),g(j),g(k))=i_0$ for all $i<j<k$. For any~$i\in\mathbb N$, another application of Higman's lemma tells us that $\preceq_{g(i)}$ is a well partial order. This yields indices $i<j<k$ with
\begin{equation*}
o_{g(i)}(g(j))\preceq_{g(i)} o_{g(i)}(g(k)).
\end{equation*}
Hence we get $i_0=c(g(i),g(j),g(k))=1$, which means that the last inequality holds for all indices~$i<j<k$. We want to conclude that $u_{g(0)},u_{g(1)},\dots$ is the directed subsequence required by the proposition. It suffices to show that
\begin{equation*}
o_i(j)\preceq_i o_i(k)\quad\Rightarrow\quad h_{ik}=h_{jk}\circ h_{ij}
\end{equation*}
holds for all $i<j<k$. One readily verifies $h_{jk}\circ h_{ij}\in\hig(u_i,u_k)$. We have already observed that $h_{ik}=h[u_i,u_k]$ is minimal in this set, so that we have
\begin{equation*}
h_{ik}(l)\leq h_{jk}\circ h_{ij}(l)\quad\text{for all $l<[u_i]$}.
\end{equation*}
To establish the converse inequality for a given $l<[u_i]$, recall that $o_i(j)\preceq_i o_i(k)$ entails the inequality
\begin{equation*}
u_j\!\restriction\!(h_{ij}(l)+1)\leq_H u_k\!\restriction\!(h_{ik}(l)+1).
\end{equation*}
The latter is witnessed by a strictly increasing $h:\{0,\dots,h_{ij}(l)\}\to\{0,\dots,h_{ik}(l)\}$ such that we have $u_j(m)\leq_X u_k(h(m))$ for all $m\leq h_{ij}(l)$. By induction on~$m$ one shows~$h_{jk}(m)\leq h(m)$. In particular we get $h_{jk}\circ h_{ij}(l)\leq h(h_{ij}(l))\leq h_{ik}(l)$.
\end{proof}

By combining the previous propositions, we obtain our main technical result:

\begin{theorem}[$\aca_0$]\label{thm:W_D-wpo}
If $D$ is a WO-dilator, then $W_D$ is a WPO-dilator.
\end{theorem}
\begin{proof}
In view of Proposition~\ref{prop:W_D-PO-dilator} it suffices to show that $\overline W_D(X)$ is a well partial order whenever the same holds for~$X$. Here $\overline W_D$ is the class-sized extension of the coded PO-dilator~$W_D$, as constructed in Section~\ref{sect:formalize-po-dilators}. According to Theorem~\ref{thm:reconstruct-class-dilator} we have $\overline W_D(X)\cong W_D(X)$. While the general version of the cited theorem cannot be formulated in reverse mathematics, the present instance is provable in~$\rca_0$, as in the case of Example~\ref{ex:coded-class-iso}. Given a wpo~$X$, it thus remains to show that~$W_D(X)$ is a wpo. For this purpose we consider an infinite sequence
\begin{equation*}
(u_0,\sigma_0),(u_1,\sigma_1),\ldots\subseteq W_D(X).
\end{equation*}
By Proposition~\ref{prop:directed-subsequence} we may assume that $u_0\leq_H u_1\leq_H\ldots\subseteq\emb(X)$ is directed. Then Proposition~\ref{lem:directed-sequence-limit} yields a family of quasi embeddings $u^i:[u_i]\to\emb(X)$ such that we have $u^j\circ h[u_i,u_j]=u^i$ for all $i<j$. We can now apply Proposition~\ref{prop:extension-wpo} with $Z=\emb(X)$, which is a well partial order by Higman's lemma. This yields indices $i<j$ with $(u_i,\sigma_i)\leq_{W_D(X)}(u_j,\sigma_j)$, so that the sequence above is good.
\end{proof}

Finally, we are able to deduce our main equivalence:

\begin{samepage}
\begin{theorem}\label{thm:main}
The following are equivalent over $\rca_0$ extended by the chain-antichain principle:
\begin{enumerate}
\item the principle of $\Pi^1_1$-comprehension,
\item the uniform Kruskal theorem: if $W$ is a normal WPO-dilator, then $\T W$ is a well partial order,
\item any normal WPO-dilator has a well partial ordered Kruskal fixed point.
\end{enumerate}
\end{theorem}
\end{samepage}
\begin{proof}
From Theorem~\ref{thm:Pi11-to-uniform-Kruskal} we know that~(1) implies~(2). The implication from~(2) to~(3) holds because $\T W$ is a Kruskal fixed point of~$W$, due to Theorem~\ref{thm:initial-fixed-point}. The same theorem tells us that any other Kruskal fixed point~$X$ of~$W$ admits a quasi embedding~$f:\T W\to X$. The latter ensures that $\T W$ is a well partial order if the same holds for~$X$, which shows that~(3) implies~(2). To complete the proof we assume~(2) and deduce~(1). In view of Lemma~\ref{lem:uniform-Kruskal_ACA} (which uses the chain-antichain principle) we can work over~$\aca_0$. Freund~\cite[Theorem~4.3]{freund-computable} has shown that \mbox{$\Pi^1_1$-comprehension} is equivalent to a computable Bachmann-Howard principle, which asserts that $\vartheta(D)$ is well founded for any WO-dilator~$D$. So it suffices to establish this principle. Given a WO-dilator~$D$, we consider the normal PO-dilator~$W_D$ from Proposition~\ref{prop:W_D-PO-dilator}. Lemma~\ref{lem:WO-dilator-monotone} (originally due to Girard) ensures that~$D$ is monotone. Then Proposition~\ref{prop:D-into-W_D} yields a quasi embedding~\mbox{$\nu:D\Rightarrow W_D$}. The latter can be transformed into a quasi embedding $f:\vartheta(D)\to\T W_D$, due to Theorem~\ref{thm:quasi-embedding}. We now invoke Theorem~\ref{thm:W_D-wpo} to learn that $W_D$ is a WPO-dilator. By the uniform Kruskal theorem from~(2) it follows that~$\T W_D$ is a well partial order. Due to the quasi embedding $f$, this implies that $\vartheta(D)$ is well founded, as required by the computable Bachmann-Howard principle.
\end{proof}

We do not know if the equivalence holds without the chain-antichain principle. The obvious attempt at a positive answer would go via a variant of the transformation $X\mapsto 1+Z\times X$ from Lemma~\ref{lem:uniform-Kruskal_ACA}. It would certainly be interesting to investigate $\T W$ in cases where $\rca_0$ proves that $X\mapsto W(X)$ preserves well partial orders.

\bibliographystyle{amsplain}
\bibliography{Kruskal_general}

\end{document}